\documentclass[a4paper,11pt]{amsart}
\usepackage[english]{babel}
\usepackage[utf8]{inputenc}
\usepackage[margin=2.5cm]{geometry}
\usepackage{hyperref}
\usepackage{amsmath}
\usepackage{amsthm}
\usepackage{amsfonts}
\hypersetup{colorlinks,urlcolor=blue,citecolor=lightblue,linkcolor=lightpink,
linktocpage,bookmarksnumbered}
\usepackage{graphicx}
\usepackage{import}
\usepackage{xifthen}
\usepackage{amsrefs}
\usepackage{enumerate}
\usepackage{pdfpages}
\usepackage{transparent}
\usepackage{parskip}
\usepackage{tikz-cd}
\usetikzlibrary{babel}
\usepackage{amssymb}
\usepackage{float}
\usepackage[capitalize]{cleveref}
\usepackage{adjustbox}

\docsvlist{N,Z,Q,R,C,F,P,RP,S,D,T,H,O}

\docsvlist{A,C,F,O,B,S,M,F,D,I,H,O}

\docsvlist{F,H,K,G,f,g,h,k}

\docsvlist{F,K,G,H,C,f,k,g,h,c,m,n,Q}

\newtheorem{teo}{\bfseries Theorem}[section]

\newtheorem{defi}[teo]{\bfseries Definition}
\newtheorem{lema}[teo]{\bfseries Lemma}
\newtheorem{prop}[teo]{\bfseries Proposition}
\newtheorem{cor}[teo]{\bfseries Corollary}

\theoremstyle{remark}
\newtheorem{remark}[teo]{\bfseries Remark}

\DeclareMathOperator{\res}{res}

\begin{document}

\definecolor{lightblue}{rgb}{0.35, 0.55, 1}
\definecolor{lightpink}{rgb}{1,0.3,0.6}

\title{The Coefficients of the $C_p$-Equivariant Geometric Complex Cobordism}
\author{Sebastian G\'omez Rend\'on}

\begin{abstract}
    We give a complete calculation of the cobordism ring of stably almost complex $C_p$-manifolds in terms of generators and relations. We also compare these generators with the geometrically-defined generators obtained by Kosniowski.
\end{abstract}

\maketitle

\begin{figure}[ht]
        \centering
        \includegraphics[width = 0.9\textwidth]{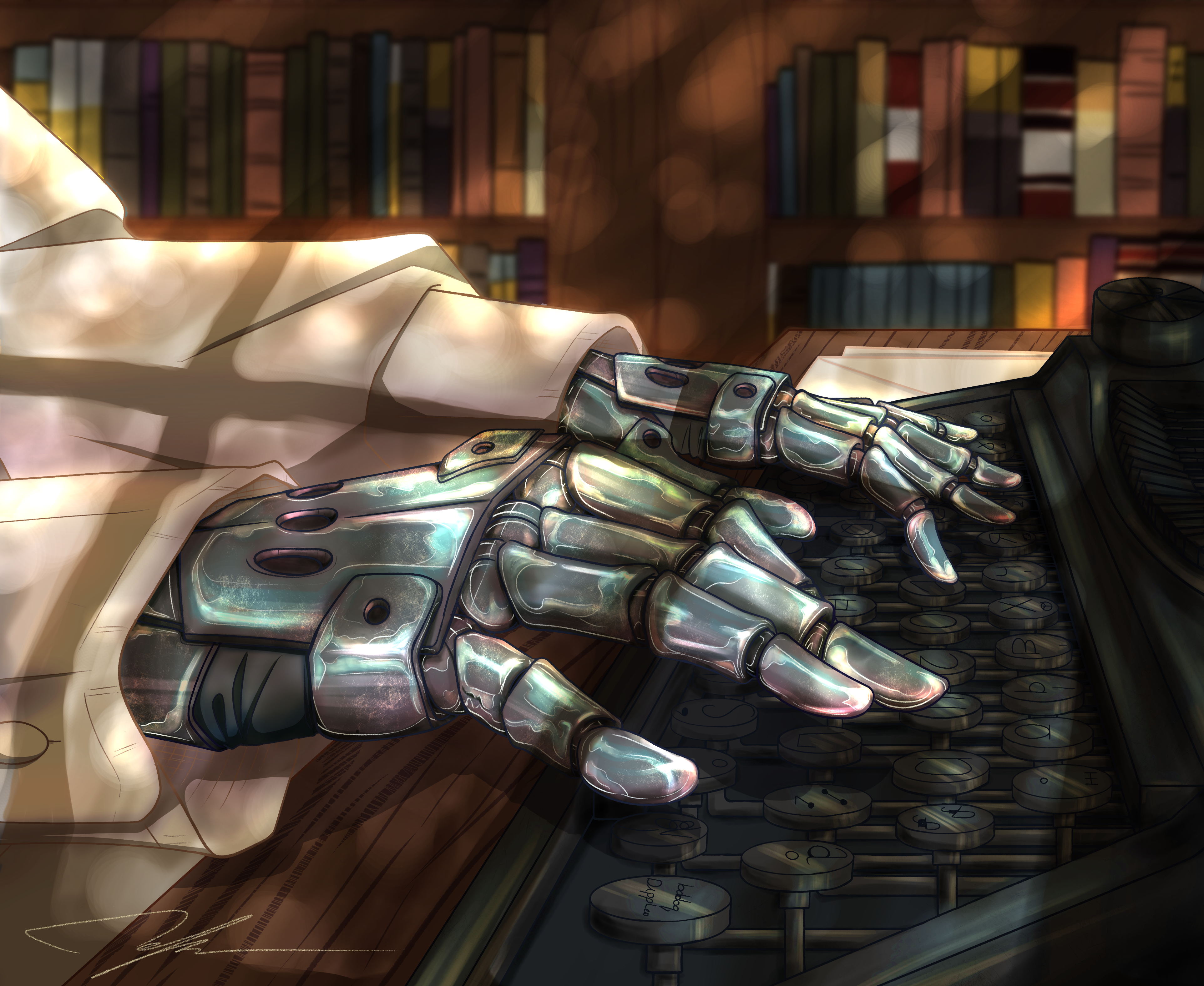}
        \caption{\textit{The Beauty in the Process of Creation}, Nicole Escanes Garc\'ia, 2026.}
\end{figure}

\section{Introduction}
For a compact Lie group $G$, a stably almost complex $G$-manifold $M$ is a closed smooth $G$-manifold such that its tangent bundle admits a complex structure after stabilization. The set of these manifolds, under the cobordism relation, forms $\Omega^{G,\text{geo}}_\ast$, the $G$-equivariant geometric complex cobordism ring. On the other hand, homotopical $G$-equivariant complex cobordism $MU^G$ was defined by tom Dieck \cite{dieck1970bordism} as a genuine $G$-spectrum. 

It is well known that the Pontrjagin-Thom map  $PT:\Omega^{G,\text{geo}}_\ast\to MU^G_\ast$ fails to be an isomorphism if $G\neq e$, since the usual transversality arguments fail in general for equivariant maps. The non-existence of an isomorphism between $\Omega^{G,\text{geo}}_\ast$ and $MU^G_\ast$ is even more fundamental as $MU^G$ has non-zero elements in negative degrees, namely the Euler classes of the irreducible complex $G$-representations. In \cite{comezana1996calculations}, Comeza\~{n}a showed that for $G$ abelian, the Pontrjagin-Thom map is injective, making $\Omega^{G,\text{geo}}_\ast$ an $MU_\ast$-subalgebra of $MU^G_\ast$.

Overall, the explicit structure of $\Omega^{G,\text{geo}}_\ast$ has not been well-understood. For $G=C_p$,  the
cyclic group of order $p$ where $p$ is any prime, Kosniowski \cite{kosniowski1976generators} found geometrically-defined generators for $\Omega^{G,\text{geo}}_\ast$ as an $MU_\ast$-algebra, without relations among them. More recently, Carlisle \cite{carlisle2022complex} described generators and relations for $\Omega^{C_2,\text{geo}}_\ast$ as an $MU_\ast$-algebra. The purpose of this paper is to give a complete calculation of $\Omega^{G,\text{geo}}_\ast$, in terms of generators and relations, for $G=C_p$. We achieve this by constructing a genuine $C_p$-spectrum $\Omega^{C_p}$ that represents $C_p$-equivariant geometric complex cobordism, that is we have an isomorphism $\Omega^{C_p,\text{geo}}_\ast\cong \Omega^{C_p}_\ast$. In fact, we can construct a genuine $G$-spectrum $\Omega^{G}$ such that $\Omega^{G,\text{geo}}_\ast\cong \Omega^{G}_\ast$ for $G$ finite abelian. This is analogous to a similar result for unoriented geometric cobordism, originally due to Wasserman \cite{wasserman1969equivariant}, and shown by Schwede \cite{schwede2018global} in the language of global spectra. 

The other key ingredient is an explicit presentation of $MU^{C_p}_\ast$ in terms of generators and relations, which was given by Strickland \cite{strickland} for $p=2$, and by Hu \cite{hu2025equivariant} for all $p$. This allows us to express $\Omega^{C_p}_\ast$ as a pullback of $MU^{C_p}_\ast$ and the geometric fixed points $\Phi^{C_p}\Omega^{C_p}_\ast$, $\Phi^{C_p}MU^{C_p}$. However, an alternative presentation of $MU^{C_p}_\ast$ is needed to take account of an additional involution in the geometric fixed points of the Pontrjagin-Thom map, which arises from interchanging the roles of the
stable normal bundle of a manifold $M$ and the normal bundle of the inclusion $M^G\to M$ \cite{dieck1970bordism}. We obtain this new presentation using methods similar to those of \cite{strickland} and \cite{hu2025equivariant}.

\textbf{Main result.} The following is the main result of this paper.

\textit{Theorem} \ref{thm:generators-geometric-cobordism}. For any prime $p$, as an $MU_\ast$-algebra, $\Omega^{C_p}_\ast$ has generators
    \[d_{l,j}^{(i)}, q_j \text{ with } |d_{l,j}^{(i)}|=2(l+j+1), |q_j|=2(j-1)\]
for $1\leq i\leq p-1$ and $0\leq l,j$, which are subject of the following relations
    \begin{align*}
        d^{(i)}_{l,j+1}(d^{(i')}_{k,s}-t^{(i')}_{k,s}) & =d^{(i')}_{k,s}(d^{(i)}_{l,j}-t^{(i)}_{l,j}) \\
    d^{(i)}_{l,j+1}(q_k-c_k) & =q_{k+1}(d^{(i)}_{l,j}-t^{(i)}_{l,j}) \\
    q_{j+1}(q_k-c_k) & =q_{k+1}(q_j-c_j) \\
    q_0 & =0\\
    d_{0,j}^{(1)} & =0 \text{ for all } j\geq 0.
    \end{align*}

Here, $c_k \in MU_\ast$ are the coefficients of the $p$-series $[p]u$ of the universal formal group law $F$, and $t_{l,j}^{(i)} \in MU_\ast$ is the coefficient of $x^{l}u^{j}$ in the series $(x +_F [i]u)^{-1}$ where the inverse is taken modulo $[p]u$. For $p = 2$, this is Theorem 3.1 in \cite{carlisle2022complex}. 

The organization of this paper is as follows. In Section \ref{section2}, we review
the basics of equivariant stably almost complex cobordism. We define the $C_p$-equivariant Thom spectrum $\Omega^{C_p}$ and show that its coefficient ring is isomorphic to $\Omega^{C_p,\text{geo}}_\ast$. We also establish the pullback diagram needed to calculate $\Omega^{C_p}_\ast$. In Section \ref{section:5}, we obtain an alternative
presentation of $MU^{C_p}_\ast$, which is better suited toward computing $\Omega^{C_p}_\ast$ as a $MU_\ast$-subalgebra. Section \ref{section4} states and proves the main theorem
that describes $\Omega^{C_p}_\ast$ explicitly. Finally, in Section \ref{section5}, we compare the
generators of $\Omega^{C_p}_\ast$ with those previously obtained by Kosniowski.

\textbf{Acknowledgments.} The author would like to thank Po Hu for suggesting this problem as well for numerous and valuable comments on both the organization and the mathematical content of this work.
\vspace{0.5 cm}

\section{Equivariant Complex Cobordism and Thom Spectra}\label{section2}
We review some standard definitions and results about equivariant complex cobordism. Let $G$ be a compact Lie group. We start with the definition of a stably almost complex $G$-manifold (see e.g \cite{hanke2005geometric})
\begin{defi}
Let $M$ be a closed smooth $G$-manifold. $M$ has a \textbf{stable almost complex} $G$-structure if there is an  equivariant bundle map
\[\phi: TM\oplus \underline{\R}^{k}\to TM\oplus \underline{\R}^{k}\]
for some $k\geq 0$, such that $\phi^2=-id_{TM\oplus \underline{\R}^{k}}$, where $\underline{\R}^{k}=M\times \R^k$ has trivial action on the fibers. A \textbf{stably almost complex} $G$-manifold is a closed smooth G-manifold together with a stable almost complex structure.
\end{defi}

It follows from this definition that the normal bundle $\nu(M^G\hookrightarrow M)$ has a complex $G$-bundle structure without requiring any stabilization. The definition of cobordism on stably almost complex $G$-manifolds follows the same lines as in the nonequivariant case.

\begin{defi}
Let $X$ be a $G$-space. The $n$-th \textbf{geometric $G$-equivariant complex cobordism} group of $X$, denoted by $\Omega^{G,\text{geo}}_n(X)$, is the group of cobordism classes of $n$-dimensional stably almost complex $G$-manifolds $[M\to X]$. 
\end{defi}

We recall now the definition of the homotopical $G$-equivariant complex cobordism spectrum as introduced in \cite{dieck1970bordism}. Recall that a complete complex universe of a compact Lie group $G$ is a complex $G$-representation $\mathcal{U}$ which contains infinitely many copies of each irreducible complex representation of $G$. If $\mathcal{U}$ is a complete complex $G$-universe, for each finite dimensional subrepresentation $V<_{\text{rep}} \mathcal{U}$ let
\[Gr^{\mathcal{U}}(V)=\{W< V\oplus \mathcal{U} \mid |W|_\C=|V|_\C\}\]
be the Grassmannian of all complex linear subspaces of $V\oplus \mathcal{U}$ of dimension $|V|_\C$. Define
\[MU^G(V)=\text{Th}(\xi^{\mathcal{U}}(V)\to Gr^{\mathcal{U}}(V))\]
where $\xi^{\mathcal{U}}(V)\to Gr^{\mathcal{U}}(V)$ is the tautological $|V|_\C$-dimensional complex bundle. Then $MU^G$ is defined as the $G$-spectrum given by the prespectrum consisting of the spaces $MU^G(V)$. 

We define another $G$-spectrum $\Omega^G$, which will represent the geometric $G$-equivariant complex bordism $\Omega^{G,\text{geo}}_\ast$. As with $Gr^\mathcal{U}(V)$, let 
\[Gr^{\C^\infty}(V)=\{W<V\oplus \C^\infty\mid |W|_\C=|V|_\C\}\]

\begin{defi}
For a finite dimensional subrepresentation $V<_{\text{rep}}\mathcal{U}$ let
\[\Omega^G(V)=\text{Th}(\xi^{\C^\infty}(V)\to Gr^{\C^\infty}(V))\]
where $\xi^{\C^\infty}(V)\to Gr^{\C^\infty}(V)$ is the tautological $|V|_\C$-dimensional complex bundle. Then  $\Omega^G$ is defined as the $G$-spectrum given by the prespectrum consisting of the spaces $\Omega^G(V)$.
\end{defi}
The inclusion of universes $j:\C^\infty \to \mathcal{U}$ induces a map of $G$-spectra
\begin{align*}\tag{1}\label{eq:01}
    j:\Omega^G\to MU^G
\end{align*}
We review now the construction of the Pontrjagin-Thom map $PT:\Omega^{G,\text{geo}}_\ast\to MU^G_\ast$ (see \cite{hanke2005geometric}). Given a class $[M]\in \Omega^{G,\text{geo}}_{2k}$, we can embed $M$ in a $G$-representation $V\oplus \C^{k}$. We can choose $V$ large enough so that the normal bundle $\nu_M^{V\oplus \C^k}$ has a complex structure. By collapsing the complement of the disc bundle of $\nu_M^{ V\oplus \C^k}$, we get a map
\[S^{V+2k}\to Th(\nu_M^{ V\oplus \C^k})\]
Consider the classifying map $\nu_M^{ V\oplus \C^k} \to \xi^{\C^k}(V)$, as well as the inclusions
\[\xi^{\C^k}(V)\to \xi^{\C^\infty}(V)\xrightarrow{j} \xi^{\mathcal{U}}(V)\]
given by the inclusions $\C^k\to \C^\infty\to \mathcal{U}$. By collapsing the complement of the disc bundles, we obtain maps
\[S^{V+2k}\to  Th(\nu_M^{ V\oplus \C^k})\to \text{Th}( \xi^{\C^\infty}(V))\to \text{Th}(\xi^{\mathcal{U}}(V)).\]
Then 
\[PT([M])=\text{colim}_V [S^{V+2k}\to\text{Th}(\xi^{\mathcal{U}}(V))]\in MU_{2k}^G.\]
From this, it is clear that $PT:\Omega^{G,\text{geo}}_\ast\to MU^G_\ast$ factors through $\Omega^G$
\[PT:\Omega^{G,\text{geo}}_\ast\to\Omega^G_\ast \to MU^G_\ast.\]
From now on, let $G$ be a finite abelian group, and let $\C_1,\ldots, \C_{n}$ be the nontrivial irreducible complex representations of $G$, which are all
1-dimensional. We recall the Euler class of the representation $\C_i$. Considering each $\C_i$ as a vector bundle over a point, we get the commutative diagram
\[\begin{tikzcd}\tag{2}\label{diagram:Euler-classes}
\C_i \arrow[r] \arrow[d] & \xi^{\mathcal{U}}(\C_i) \arrow[d] \\
\ast \arrow[r]           & Gr^{\mathcal{U}}(\C_i)           
\end{tikzcd}\]
Taking Thom spaces, we get a map $S^{\C_i}\to MU^G(\C_i)$ and by precomposing with the inclusion of the fixed points $S^0\to S^{\C_i}$ we get a map
\[S^0\to MU^G(\C_i)\]
This map defines a class $u_i\in MU^G_{-2}$ which is the Euler class of the representation $\C_i$. 

The following proposition is analogous to the calculation of $\Phi^{C_p}MU^{C_p}$ in \cite{kriz1999z}.

\begin{prop}\label{prop:geo-fixed-points}
    Let $G$ be a finite abelian group with $|G|=n+1$, then
    \begin{equation*}\tag{3}\label{eq:03}
        \Phi^{G}\Omega^{G} = \bigvee_{m_1,\ldots,m_{n}\geq 0} MU\wedge \Sigma^{2m_1}BU(m_1)_+\wedge\cdots\wedge \Sigma^{2m_{n}} BU(m_{n})_+.
    \end{equation*}
\end{prop}

\begin{proof}
Since $\Omega^{G}$ is an inclusion prespectrum, we have
\[\Phi^G\Omega^G\simeq \text{colim}_V \Sigma^{-V^G} (\Omega^G(V))^G.\]
For a given finite dimensional subrepresentation $V<_{\text{rep}}\mathcal{U}$ of complex dimension $k$, we can write $V=V_0\oplus\cdots \oplus V_n$ with $V_i\cong \C_i^{l_i}\leq_{\text{rep}}\C_i^\infty$, and $k=l_0+\cdots +l_n$. In particular, $V_0=V^G$. On the base space, $(Gr^{\C^\infty}(V))^G$ is the subspace of $Gr^{\C^\infty}(V)$ consisting of all subrepresentations of $V\oplus \C^\infty$. For such a subrepresentation $W$, we can write $W=W_0\oplus \cdots \oplus W_n$ with $W_i\cong \C_i^{k_i}$ for some $k_i\geq 0$. Thus, $W_i\subset V_i$ for $1\leq i\leq n$ and $W_0=W^G\subset (V\oplus \C^\infty)^G=V_0\oplus \C^\infty$.
Note that a subrepresentation of $V_i$ is the same as a $\C$-linear subspace of $V_i$. Hence, 
\begin{align*}
    (Gr^{\C^\infty}(V))^G= \bigvee_{k_0+\cdots +k_n=k} Gr(k_0,V_0\oplus\C^\infty)_+
    \wedge Gr(k_1,V_1)_+ \wedge \cdots \wedge Gr(k_n,V_n)_+.
\end{align*}
where $k_0 \geq 0$, and $0 \leq k_i \leq l_i$ for $1 \leq i \leq n$. For $W\in (Gr^{\C^\infty}(V))^G$, the fixed points in the fiber of $\xi^{\C^\infty}(V)$ over $W$ consist of $(w,W)$ with $w\in W_0$. Therefore,
\begin{align*}
    \Sigma^{-V^G}(\Omega^G(V))^G= \bigvee_{k_0+\cdots +k_n=k} \Sigma^{-2l_0}MU(k_0) \wedge Gr(k_1,V_1)_+\wedge\cdots \wedge Gr(k_n,V_n)_+.
\end{align*}
Since $\sum l_i=\sum k_i=k$, we get $l_0=k_0+\sum (k_i-l_i)$. Thus, taking the colimit over $V$ in the formula above, we get
\begin{equation*}
     \Phi^G \Omega^G\simeq \text{colim}_V \bigvee_{k_0+\cdots +k_n=k} \Sigma^{-2k_0} MU(k_0) \wedge \Sigma^{2(l_1-k_1)} Gr(k_1,V_1)_+\wedge\cdots \wedge \Sigma^{2(l_n-k_n)}Gr(k_n,V_n)_+.
\end{equation*}
Now we identify the maps in the colimit. Given an inclusion $V\hookrightarrow V'$, the map $\Sigma^{-V}(\Omega^G(V))^G\to \Sigma^{-V'}(\Omega^G(V'))^G$
comes from desuspending the map resulting from taking fixed points and one point compactification of the map
\begin{align*}   \tag{4}\label{eq:inclusions}
    \xi^{\C^\infty}(V)\to \xi^{\C^\infty}(V') \\
    (w,W)\mapsto ((w,0),W\oplus (V'-V)).
\end{align*}
It suffices to consider the case $V'=V\oplus \C_i$ with $0\leq i\leq n$. As before, we write $V'=V_0'\oplus\cdots \oplus V_n'$ with $V_i'\cong \C_i^{l'_i}$. Thus, $l'_j=l_j$ for $j\neq i$ and $l'_i=l_i+1$. For the case $V'=V\oplus \C_0$, after taking fixed points and compactification, the map \eqref{eq:inclusions} sends the wedge summand corresponding to the partition $k_0,\ldots,k_n$ of $|V|_\C$ to the wedge summand corresponding to the partition $k_0+1,k_1,\ldots, k_n$ of $|V'|_\C$. On such summands, this map on the zeroth term of the smash product is given by the analog of \eqref{eq:inclusions}
\[\Sigma^{-2k_0}MU(k_0)\to \Sigma^{-2(k_0+1)}MU(k_0+1)\]
and is the identity in the other smash product factors. For the case $V'=V\oplus \C_i$ for fixed $1\leq i\leq n$, the map in the colimit sends the wedge summand corresponding to the partition $k_0,\ldots,k_i,\ldots k_n$ of $|V|_\C$ to the wedge summand corresponding to $k_0,\ldots,k_i+1,\ldots, k_n$ of $|V'|_\C$. On such summands, this map on the $i$-th term of the smash product is given by the inclusion
\[\Sigma^{2(l_i-k_i)}Gr(k_i,V_i)\to \Sigma^{2(l_i+1-k_i-1)}Gr(k_i+1,V_i\oplus \C_i)\]
and is the identity map in the other smash product factors. In particular, each $l_i-k_i$ is constant through the colimit.

Recall that for an $l$-dimensional complex inner product space $V$, and for $k\leq l$ there is a natural homeomorphism
\begin{align*} \tag{5}\label{eq:inverse-H-space}
     i_{k,l}: Gr(k,V)\to Gr(l-k,V) \\
    W\mapsto V-W
\end{align*}
Using these maps in our context, we can write $\Phi^G\Omega^G$ as
\begin{equation*}
     \text{colim}_V\bigvee_{k_0+\cdots +k_n=k} \Sigma^{-2k_0} MU(k_0) \wedge \Sigma^{2(l_1-k_1)} Gr(l_1-k_1,V_1)_+\wedge\cdots \wedge \Sigma^{2(l_n-k_n)}Gr(l_n-k_n,V_n)_+. 
\end{equation*}
Setting $l_i-k_i=m_i$, we get the desired result.
\end{proof}

The geometric fixed points of $\Omega^{G,\text{geo}}_\ast$, on the other hand, are given by the following well known result.

\begin{lema}\cite{comezana1996calculations}*{Proposition 4.2}\label{lema:identify-geo-fixed-points}
    For any finite abelian group $G$ with $|G|=n+1$ we have
    \[\Phi^G \Omega^{G,\text{geo}}\cong  \bigoplus_{m_1,\ldots,m_{n}\geq 0}MU_{\ast-2(m_1+\ldots+m_{n})}( BU(m_1)\times\cdots\times BU(m_n)).\]
\end{lema}
Recall that there is a map $\iota: BU\to BU$ which classifies the stable additive inverse of a bundle. By an abuse of notation we also denote by $\iota:BU_+\wedge \ldots\wedge BU_+\to BU_+\wedge \ldots\wedge BU_+$ the map which is the $n$-fold product of the map $\iota: BU\to BU$. Recall that the inclusion of universes $\C^\infty\to \mathcal{U}$ induces a map $j:\Omega^G\to MU^G$ of $G$-spectra \eqref{eq:01}. 

\begin{lema}\label{lema:geo-fix-j}
    The map induced on geometric fixed points $\Phi^G j_\ast: \Phi^G\Omega^G_\ast \to \Phi^G MU^G_\ast$ is
    \[\Phi^G j_\ast= \iota_\ast\circ k_\ast\]
    where $k_\ast:\Phi^G\Omega^G_\ast \to \Phi^G MU^G_\ast$ is the map induced from the inclusion 
    \[BU(m_1)_+\wedge\ldots\wedge BU(m_n)_+\to BU_+\wedge \ldots\wedge BU_+\]
\end{lema}

\begin{proof}
    By \cite{dieck1970bordism} (see also \cite{kriz1999z}*{(2.9)}) we have
    \begin{equation*}\tag{6}\label{eq:6}
        \Phi^G MU^G =\bigvee_{m_1,\ldots,m_{n}\in \Z} MU\wedge \Sigma^{2m_1}BU_+\wedge\cdots\wedge \Sigma^{2m_{n}} BU_+.
    \end{equation*}
    In particular, for a $k$-dimensional complex representation $V$ we have
    \[\scalebox{0.95}{$\displaystyle \Sigma^{-V^G}(MU^G(V))^G= \bigvee_{k_0+\ldots+k_n=k}MU(k_0)\wedge \Sigma^{2(l_1-k_1)}Gr(k_1,V_1\oplus \C^\infty_1)_+\wedge\ldots\wedge \Sigma^{2(l_n-k_n)}Gr(k_n,\C^\infty_n)_+.$}\]
where $k_0,\ldots,k_n\geq 0$. By Proposition \ref{prop:geo-fixed-points}, we have a map $(\Omega^G(V))^G\to (MU^G(V))^G$ which in the first smash factor is the identity $MU(k_0)\to MU(k_0)$, and on the other smash factors is the inclusion $Gr(k_i,V_i)\to Gr(k_i,\C^\infty_i)$. Recall that we used the homeomorphisms $i_{k,l}$ \eqref{eq:inverse-H-space} to identify $\Phi^G\Omega^G$. Taking them into account, we obtain the following commutative diagram
\begin{center}
\begin{adjustbox}{width=\textwidth,center}
\begin{tikzcd}
\bigvee MU(k_0)\wedge Gr(k_1,V_1)_+\wedge\ldots\wedge Gr(k_n,V_n)_+ \arrow[d] \arrow[r] &
\bigvee MU(k_0)\wedge Gr(l_1-k_1,V_1)_+\wedge\ldots\wedge Gr(l_n-k_n,V_n)_+ \arrow[d] \\
\bigvee MU(k_0)\wedge Gr(k_1,\C^\infty_1)_+\wedge\ldots\wedge Gr(k_n,\C^\infty_n)_+ \arrow[r] &
\bigvee MU(k_0)\wedge Gr(l_1-k_1,\C^\infty_1)_+\wedge\ldots\wedge Gr(l_n-k_n,\C^\infty_n)_+
\end{tikzcd}
\end{adjustbox}
\end{center}
After taking desuspensions and passing to the colimit, the left vertical map corresponds to $k_\ast$, the bottom map is $\iota_\ast$, and the right vertical map is $\Phi^G j_\ast$.
\end{proof}

Recall that $MU_\ast(BU(1)_+)=MU_\ast\{b_0',b_1',b_2',\ldots\}$ with $|b_l'|=2l$ and $b_0'=1$. Hence, in \eqref{eq:03} for each $1\leq i\leq n$ and $m_i\geq 0$, $MU_\ast(\Sigma^{2m_i}BU(m_i))$ is the $MU_\ast$-module generated by the degree $m_i$ monomials in $u_i^{-1}b_0',u_i^{-1}b_1',\ldots$ where $u^{-1}_i$ is the inverse of the Euler class \eqref{diagram:Euler-classes}. Let $b_l^{(i)\prime}$ denote the image of $b_l'$ under the inclusion of $BU(1)$ into the $i$-th copy of $BU$ in \eqref{eq:6}. We get
\[\Phi^G\Omega^G_\ast=MU_\ast[u_i^{-1},u_i^{-1}b_l^{(i)\prime}\mid 1\leq l, 1\leq i\leq n] \text{ with } |u_i^{-1}|=2 \text{, } |b_l^{(i)\prime}|=2l.\]
Similarly,
\[\Phi^GMU^G_\ast=MU_\ast[u_i^{\pm 1}, u_i^{-1}b_{l}^{(i)\prime}\mid 1\leq l, 1\leq i\leq n] \text{ with } |u_i|=-2 \text{, }|b_l^{(i)\prime}|=2l\]
Note that $\Phi^G \Omega^G_\ast$ and $\Phi^G\Omega^{G,\text{geo}}_\ast$ have the same presentation as $MU_\ast$-algebras. We still need to see that $\Phi^G PT: \Phi^G\Omega^{G,\text{geo}}_\ast \to \Phi^G \Omega^G_\ast$ is the identity map.

\begin{prop}\label{prop: commutative-diagram-geo-fixed-pts}
For a finite abelian group $G$, the following diagram is commutative 
   \[\begin{tikzcd}
{\Omega^{G,\text{geo}}_\ast} \arrow[r, "{\Phi_{\Omega^{G,\text{geo}}}}"] \arrow[d, "PT"'] & {MU_\ast[u_i^{-1},u_i^{-1}b_l^{(i)\prime}]} \arrow[d, "\Phi^G PT=id"]          \\
\Omega^{G}_\ast \arrow[d, "j_\ast"'] \arrow[r, "\Phi_{\Omega^{G}}"]                       & {MU_\ast[u_i^{-1},u_i^{-1}b_l^{(i)\prime}]} \arrow[d, "\iota_\ast\circ k_\ast"] \\
MU^G_\ast \arrow[r, "\Phi_{MU^{G}}"]                                                      & {MU_\ast[u_i^{\pm 1},u_i^{-1}b_l^{(i)\prime}]}                             
\end{tikzcd}\tag{7}\label{diagram:big-square}\]
The horizontal maps are the canonical maps for each spectrum into its geometric fixed points, and $k_\ast:MU_\ast[u_i^{-1},u_i^{-1}b_l^{(i)\prime}]\to MU_\ast[u_i^{\pm1},u_i^{-1}b_l^{(i)\prime}] $ is the inclusion. Moreover, $\Phi^G PT =id$.
\end{prop}

\begin{proof}
The lower diagram commutes by Lemma \ref{lema:geo-fix-j}. In \cite{dieck1970bordism}*{Proposition 4.1} it is proven that the outer square commutes, and that the right vertical composition equals $\iota_\ast\circ k_\ast$. Therefore, we get $\iota_\ast\circ k_\ast=\iota_\ast\circ k_\ast\circ \Phi^G PT$. Since $\iota_\ast$ is an involution and $k_\ast$ is the inclusion we conclude that  $\Phi^G PT =id$.
\end{proof}
From now on, we restrict to the case $G=C_p$. The isotropy separation sequence in this case is given by
\[(EC_{p})_+\to S^0\to \widetilde{EC_{p}}.\]
Here $\widetilde{EC_{p}}$ denotes the unreduced suspension. We obtain the following diagram,
 \[\begin{tikzcd}
\Omega^{C_p,\text{geo}}_\ast(EC_p)_+ \arrow[r] \arrow[d, "PT(EC_p)_+"'] & \Omega^{C_p,\text{geo}}_\ast \arrow[r] \arrow[d, "PT"'] & {MU_\ast[u_i^{-1},u_i^{-1}b_l^{(i)\prime}]} \arrow[d, "id"] \\
\Omega^{C_p}_\ast(EC_p)_+ \arrow[r]                                  & \Omega^{C_p}_\ast \arrow[r]                          & {MU_\ast[u_i^{- 1},u_i^{-1}b_l^{(i)\prime}]}                                 
\end{tikzcd}\]
By \cite{dieck1970bordism}, transversality holds for free manifolds. So the left vertical map in the diagram above is an isomorphism. By the Five Lemma, we get
\begin{teo}\label{teo:PT-iso}
     The Pontrjagin-Thom map $PT:\Omega^{C_p,\text{geo}}_\ast\to \Omega^{C_p}_\ast$ is an isomorphism.
\end{teo}

\begin{remark}
    In \cite{schwede2018global}*{Theorem 6.2.33}, Schwede proves an analogous result for unoriented $G$-manifolds, for $G$ a product of a torus and  finite group. Schwede's result reinterprets a previous result of Wasserman \cite{wasserman1969equivariant}*{Theorem 3.11}. An analogous statement for stably complex $G$-manifolds generalizing Theorem \ref{teo:PT-iso} can be shown, but this is beyond the scope of this note.
\end{remark}

Similarly, by evaluating $MU^{C_p}$ and $\Omega^{C_p}$ on the isotropy separation sequence, we obtain the following diagram. 
\[\begin{tikzcd}\tag{8}\label{eq:08}
\Omega^{C_p}_\ast(EC_p)_+ \arrow[r] \arrow[d, "j_\ast(EC_p)_+"'] & \Omega^{C_p}_\ast \arrow[r] \arrow[d, "j_\ast"'] & {MU_\ast[u_i^{-1},u_i^{-1}b_l^{(i)\prime}]} \arrow[d, "\iota_\ast\circ k_\ast"] \\
MU^{C_p}_\ast(EC_p)_+ \arrow[r]                                  & MU^{C_p}_\ast \arrow[r]                          & {MU_\ast[u_i^{\pm 1},u_i^{-1}b_l^{(i)\prime}]}                                 
\end{tikzcd}\]
    
The lower square in \eqref{diagram:big-square} is the right square in the diagram above. $MU^{C_p}$ and $\Omega^{C_p}$ are split spectra and the underlying nonequivariant map $\Omega^{C_p}\xrightarrow{j_\ast} MU^{C_p}$ is a nonequivariant equivalence. Since $EC_{p+}$ is a free $C_p$-space, we get that the map $j_\ast(EC_{p})_+$ is an isomorphism by \cite{greenlees1995equivariant}*{Theorem 3.15}. The map $PT:\Omega^{G,\text{geo}}_\ast \to MU^G_\ast$ is injective for all abelian $G$ by \cite{comezana1996calculations}*{Theorem 5.4}, so it follows from Theorem \ref{teo:PT-iso} that $j_\ast:\Omega^{C_p}_\ast \to MU^{C_p}_\ast$ is injective. By \cite{kriz1999z}*{Lemma 2.1}, the right square in \eqref{eq:08} is a pullback square.

We want to change the right vertical map in \eqref{eq:08} so that it becomes only the inclusion map. To this end, we identify $\iota_\ast(u_i^{-1}b_l^{(i)\prime})$. In the Hopf algebra structure on $MU_\ast(BU)$, the map $\iota_\ast$ is the antipode. Taking $b_0'=1$, the coproduct of this Hopf algebra is given by
\[\Delta(b_l')=\sum_{k=0}^l b_k'\otimes b_{l-k}'\]
and the product is given by the usual polynomial product. The unit $\eta: MU_\ast\to MU_\ast(BU)$ is determined by sending $1$ to $b_0'$, and the counit $\epsilon:MU_\ast(BU)\to MU_\ast$ is determined by
\[\epsilon(b_0')=1 \text{, } \epsilon(b_l')=0 \text{ for } l>0.\]
It follows that
\[\iota_\ast(b_0')=b_0'\text{, } \sum_{k=0}^l\iota_\ast(b_l')b_{l-k}'=0\]
Hence, $\iota_\ast(b_l^{\prime})=d_l'$ where $d_0'+d_1'x+d_2'x^2+\cdots$ is the multiplicative inverse of $b_0'+b_1'x+b_2'x^2+\cdots$ in $MU_\ast(BU)[[x]]$. Doing this for each copy of $BU_+$ in \eqref{eq:6}, define 
\[d_l^{(i)}=\iota_\ast(u_ib_l^{(i)\prime})\]
 So we have
\[(b_0^{(i)}+b_1^{(i)}x+\cdots)(d_0^{(i)}+d_1^{(i)}x+\cdots)=1\]
where $b_l^{(i)}=u_ib_l^{(i)\prime}$. In particular, $d_0^{(i)}=u_i^{-1}$. Note that $d_l^{(i)}$ satisfy $d_{l}^{(i)}\equiv b_l^{(i)}$ modulo $u_i^{-1}, b_{k}'$ for $k<l$. Hence, $MU_\ast[u_i^{\pm1},u_i^{-1}b_l^{(i)\prime}]=MU_\ast[u_i^{\pm1},d_l^{(i)}]$. We obtain a new presentation
\[\Phi^{C_p}MU^{C_p}_\ast=MU_\ast[u_i^{\pm 1}, d_l^{(i)}\mid 1\leq l, 1\leq i\leq p-1]\] 
with $|u_i^{-1}|=2$ and $|d_l^{(i)}|=2(1+l)$. In particular, within $MU_\ast[u_i^{\pm1},d_l^{(i)}]$ we have an isomorphism
\[\Phi^{C_p}\Omega^{C_p}=MU_\ast[u_i^{-1},u_i^{-1}b_l^{(i)\prime}]\cong \iota_\ast(MU_\ast[u_i^{-1},u_i^{-1}b_l^{(i)\prime}])=MU_\ast[u_i^{-1},d_l^{(i)}]\]
Thus, the pullback square at the bottom of \eqref{diagram:big-square} takes now the form
\[\begin{tikzcd}\tag{9}\label{eq:09}
\Omega^{C_p}_\ast \arrow[r, "\iota_\ast\circ \Phi_{\Omega^{C_p}}"] \arrow[d, "j_\ast"'] & {MU_\ast[u_i^{-1},d_l^{(i)}]} \arrow[d, "k_\ast"] \\
MU^{C_p}_\ast \arrow[r, "\Phi_{MU^{C_p}}"']               & {MU_\ast[u_i^{\pm 1},d_l^{(i)}]}                 
\end{tikzcd}\]
where $k_\ast: MU_\ast[u_i^{-1},d_l^{(i)}]\to MU_\ast[u_i^{\pm 1},d_l^{(i)}]$ is the inclusion.

\vspace{0.5cm}
\section{Calculation of \texorpdfstring{$MU^{C_p}_\ast$}{MU Cp *}}\label{section:5}

In this section, we give a new explicit presentation of $MU^{C_p}_\ast$, the lower left corner of \eqref{eq:09}. As an $MU_\ast$-algebra, $MU^{C_p}_\ast$ has been calculated in terms of generators and relations by Strickland \cite{strickland} for $p=2$, and by Hu \cite{hu2025equivariant} for general $p$, using the following pullback diagram \cite{kriz1999z}*{Theorem 1.1}
\begin{equation*}
    \begin{tikzcd}
MU^{C_p}_\ast \arrow[d] \arrow[r]       & {MU_\ast[u_i^{\pm 1}, b_l^{(i)}\mid 1\leq l, 1\leq i\leq n]} \arrow[d, "\phi"] \\
{MU_\ast [[u]] /[p]u} \arrow[r] & {(MU_\ast [[u]] /[p]u)[u^{-1}]}                              
\end{tikzcd}\tag{10}\label{diagram:pullback-kriz}
\end{equation*}
where $b_l^{(i)}=u_ib_l^{(i)\prime}$, and $[p]u$ is the $p$-series of the universal formal group law, which we denote by 
\[F(x,y)=x+_F y= \sum_{k,j\geq 0}a_{k,j}x^ky^j.\]
The right vertical map is given by
 \[\phi(u_i)=[i]u \text{, and } \phi(b_l^{(i)})=\text{Coeff}_{x^l}(x+_F [i]u)=\sum_{j\geq 0}a_{l,j}^{(i)}u^j\]
The bottom horizontal map is inverting $u$. However, the upper right corner of \eqref{diagram:pullback-kriz} is not the presentation of $\Phi^{C_p}MU^{C_p}_\ast$ that we want to use. Instead, we need an alternative presentation of $MU^{C_p}_\ast$ which uses $MU_\ast[u_i^{\pm 1},d_l^{(i)}]$ in place of $MU_\ast[u_i^{\pm 1}, b_l^{(i)}]$ in the upper right corner of \eqref{diagram:pullback-kriz}.

If $[i]u$ is the $i$-series of $u$ over $F$, we write
\[x+_F [i]u=\sum_{k,j\geq 0}a_{k,j}^{(i)}x^ku^j.\]
Note that $a_{1,0}^{(i)}=1$, $a_{k,0}^{(i)}=0$ for $k\neq 1$, and $a_{0,1}^{(i)}=i$. For the $p$-series, we write
\[[p]u=\sum_{j\geq 0}c_{j}u^j.\]
We need to replace the right upper corner of \eqref{diagram:pullback-kriz} by the presentation $\Phi^{C_p}MU^{C_p}_\ast=MU_\ast[u_i^{\pm 1}, d_l^{(i)}]$ we obtained in the last section. Hence, \eqref{diagram:pullback-kriz} becomes

\begin{equation*}\label{eq:11.1}\tag{11}
    \begin{tikzcd}
MU^{C_p}_\ast \arrow[d] \arrow[r]       & {MU_\ast[u_i^{\pm 1}, d_l^{(i)}]} \arrow[d,"\phi"] \\
{MU_\ast [[u]] /[p]u} \arrow[r] & {(MU_\ast [[u]] /[p]u)[u^{-1}]}                              
\end{tikzcd}
\end{equation*}
We calculate now $\phi(d_{l}^{(i)}$). Denote by $i^{-1}$ the smallest positive integer such that $i^{-1}i=1+k_ip$ with $k_i\geq 0$. In other words, $i^{-1}$ is the lowest positive representative of $i^{-1}\in (\Z/p)^\times$. Then in $(MU_\ast [[u]] /[p]u)[u^{-1}]$
\begin{equation*}\label{eq:f}
    u=u+_F [k_i]([p]u)=[1+k_ip]u=[i^{-1}]([i]u)=\sum_{j\geq1}a_{0,j}^{(i^{-1})}([i]u)^j=\sum_{j\geq0}a_{0,j+1}^{(i^{-1})}([i]u)^{j+1}. 
\end{equation*}
If we set 
\[\tag{12}\label{eq:f_i}f_i(u)=\sum_{j\geq0}a_{0,j+1}^{(i^{-1})}([i]u)^j\]
we get $u^{-1}f_i(u)=([i]u)^{-1}$, so $[i]u$ is invertible in $(MU_\ast [[u]] /[p]u)[u^{-1}]$. Consider $x+_F [i]u$ in $(MU_\ast [[u]] /[p]u)[u^{-1}][[x]]$. Since its constant term on $x$ is $[i]u$, $x+_F [i]u$ is invertible in $(MU_\ast[[u]]/[p]u)[u^{-1}][[x]]$. By induction, it can be seen that the lowest possible power of $u$ in $\text{Coeff}_{x^l}(x+_F [i]u)^{-1}$ is $-(l+1)$, so we write
\[(x +_F [i]u)^{-1}=\sum_{k\geq 0,j\geq -(k+1)}t_{k,j}^{(i)}x^ku^j.\]
Note that $t_{k,j}^{(i)}\in MU_\ast$ depends on the choice of $i^{-1}\in \Z$ that we set above. Moreover, we have $t_{0,-1}^{(i)}=i^{-1}$ for all $i\in (\Z/p)^\times$, and $t_{0,j}^{(1)}=0$ for all $j\geq 0$. We can compute $\phi$ in the new generators $d_l^{(i)}$ by using the fact that 
\[(b_0^{(i)}+b_1^{(i)}x+\cdots)(d_0^{(i)}+d_1^{(i)}x+\cdots)=1.\]
We get 
\[\phi(d_l^{(i)})=\text{Coeff}_{x^l}(x+_F [i]u)^{-1}=\sum_{j\geq -(l+1)}t_{l,j}^{(i)}u^j.\]
The following is the main result of this section. Recall that we write $i^{-1}$ for the lowest positive representative of $i^{-1}\in (\Z/p)^\times$, which satisfies $i^{-1}i=1+k_ip$ for some $k_i\geq 0$.
\begin{teo}\label{teo:MU-calculation}
    As an $MU_\ast$-algebra, $MU^{C_p}_\ast$ has the following generators:
    \begin{align*}
        u \text{, } d_{l,j}^{(i)} \text{, } \eta_i \text{, } q_j \text{ with } |u|=-2, |\eta_i|=0, |q_j|=2(j-1), |d_{l,j}^{(i)}|=2(l+j+1).
    \end{align*}
for $l,j\geq 0$ and $i\in (\Z/p)^\times$, with the relations:
\begin{align*}
    d_{l,j}^{(i)}-t_{l,j}^{(i)} & =ud_{l,j+1}^{(i)} \tag{13}\label{eq:relations-d}\\
    d_{0,j}^{(1)} & =0 \text{ for all $j\geq 0$} \tag{14}\label{eq:relations-d-0}\\
    q_{j}-c_{j} & =uq_{j+1} \tag{15}\label{eq:relations-q}\\
    q_0 & =0 \tag{16}\label{eq:relations-q-0}\\
    \eta_1 & =1 \tag{17}\label{eq:eta-1}\\
   \eta_i(ud_{0,0}^{(i)}+i^{-1}) & =1+k_iq_1  \tag{18}\label{eq:relations-eta-d}\\
   \eta_i q_1 & =i q_1  \tag{19}\label{eq:relations-eta-q}
\end{align*}

\end{teo}

\begin{remark}
    Relations \eqref{eq:relations-eta-d} and \eqref{eq:relations-eta-q} for $i=1$ are redundant. So, we could omit $\eta_1$ and take the mentioned relations for $i\geq 2$. Similarly, we could omit $q_0$ and change relation \eqref{eq:relations-q-0} by $uq_1=0$. We could also omit the elements $d_{0,j}^{(1)}$ and just take relations \eqref{eq:relations-d} for all $l,j$ if $i\geq 2$, and for all $l,j$ if $l\geq 1$ and $i=1$. Moreover, we can obtain a copy of the Euler classes as $u_i=u\eta_i$.
\end{remark}

The proof of Theorem \ref{teo:MU-calculation} uses the following result from commutative algebra.

\begin{prop}\cite{strickland}*{Theorem 4}\label{prop:pullback-strickland}
    Let $R$ be a commutative ring and $x\in R$ such that $x$ has bounded $x$-torsion, meaning that there is an $N$ such that
    \[\bigcup_{k\geq 0}\text{Ann}(x^k)=\text{Ann}(x^N).\]
    Then, the following square is a pullback
    \[\begin{tikzcd}
R \arrow[r] \arrow[d] & {R[x^{-1}]} \arrow[d] \\
R^{\wedge}_x \arrow[r]  & {R^{\wedge}_x[x^{-1}]} 
\end{tikzcd}\]
\end{prop}

Let $R$ be the $MU_\ast$-algebra defined by the generators and relations given of Proposition \ref{teo:MU-calculation}. We will show that $u\in R$ satisfies the hypothesis in \cref{prop:pullback-strickland}. Also, $R^{\wedge}_u\cong MU_\ast[[u]]/[p]u$, and $R[u^{-1}]\cong MU_\ast[u_i^{\pm 1},d_l^{(i)}]$ in a compatible way. We start by proving that $u$ has bounded torsion in $R$.

\begin{lema}\label{lema:bounded-torsion}
    In $R$, we have
    \[\bigcup_{k\geq 0}\text{Ann}(u^k)=\text{Ann}(u)=(q_1)\]
\end{lema}

\begin{proof}
    The proof of this result is similar to those of \cite{strickland} and \cite{hu2025equivariant}*{Lemma 2.7}. We will show that $u$ is not a zero divisor in $R/(q_1)$. Note that \eqref{eq:relations-eta-d} over $R/(q_1)$ becomes
    \[\eta_i(ud_{0,0}^{(i)})+i^{-1})=1.\]
    Fix $k\geq 2$ and set $R_k$ to be the $MU_\ast[u]$-subalgebra of $R/(q_1)$ generated by $d_{l,j}^{(i)}$, $q_j$, and $\eta_i$ for $j\leq k$. Set
    \[A_k=MU_\ast[u][d_{l,k}^{(i)},q_k\mid (i,l)\neq (1,0)].\]
For $j\leq k-1$, consider the following element of $A_k$
\[h_{l,j}^{(i)}=u(d_{l,k}^{(i)}u^{k-j}+\sum_{s=0}^{k-j-1}t_{l,j+s}^{(i)}u^{s})+i^{-1}.\]
Define a map of $MU_\ast[u]$-algebras
\begin{align*}
    \Phi: A_k[(h_{0,0}^{(i)})^{-1}\mid i\neq 1]\to R_k \\
    d_{l,k}^{(i)}\mapsto d_{l,k}^{(i)} \\
    q_k \mapsto q_k \\
    (h_{0,0}^{(i)})^{-1} \mapsto \eta_i
\end{align*}
By induction on $n\geq 1$, relation \eqref{eq:relations-d} gives
\[d_{l,j}^{(i)}=\sum_{s= 0}^{n-1}t_{l,j+s}^{(i)}u^s+u^nd_{l,j+n}^{(i)}.\]
This relation holds in $R_k$, so we get
\[d_{l,j}=\Phi(d_{l,k}^{(i)}u^{k-j}+\sum_{s=0}^{k-j-1}t_{l,j+s}^{(i)}u^{s}).\]
In particular, $\Phi(h_{0,0}^{(i)})=ud_{0,0}^{(i)}+i^{-1}$, which is consistent with $\Phi(h_{0,0}^{(i)})^{-1}=\eta_i$. A similar argument using \eqref{eq:relations-q} shows that $q_j$ is in the image of $\Phi$. We conclude that $\Phi$ is surjective. In particular, if we set
\[g_1=\sum_{s= 0}^{k-2}c_{1+s}u^s+u^{k-1}q_{k}\]
we see that $\Phi(g_1)=q_1=0$. Thus, the map $\Phi$ factors through a surjective map
\[\overline{\Phi}:A_k[(h_{0,0}^{(i)})^{-1}]/(g_1)\to R_k.\]
We want to show that $\overline{\Phi}$ is an isomorphism by giving its inverse. Define
\begin{align*}
    \overline{\Phi}^{-1}:R_k & \to A_k[(h_{0,0}^{(i)})^{-1}]/(g_1)\\
    d_{l,j}^{(i)} & \mapsto d_{l,k}^{(i)}u^{k-j}+\sum_{s=0}^{k-j-1}t_{l,j+s}^{(i)}u^{s} \\
    q_j & \mapsto \sum_{s= 0}^{k-j-1}c_{j+s}u^s+u^{k-j}q_{k}\\
    \eta_i & \mapsto h_{0,0}^{(i)}.
\end{align*}
The relations in $R_k$ are given by taking those in $R$ and setting $q_1=0$. Hence, $\overline{\Phi}^{-1}$ respects those relations by construction. It can be seen that $\overline{\Phi}^{-1}$ is indeed the inverse of $\overline{\Phi}$. 

The element $u$ is clearly not a zero divisor in $A_k[(h_{0,0}^{(i)})^{-1}]$. We claim that $u$ is not a zero divisor in $A_k[(h_{0,0}^{(i)})^{-1}]/(g_1)$. Suppose  $h=a_0+a_1 u+\cdots +a_n u^n\in A_k[(h_{0,0}^{(i)})^{-1}]$ is such that $uh\equiv 0$ mod $g_1$. Then, there is $k$ such that $uh=g_1k$. Note that $g_1$ has $p$ as constant term, and this is not a zero divisor in $MU_\ast$. By comparing coefficients, we see that $k=uk_1$ for some $k_1\in A_k[(h_{0,0}^{(i)})^{-1}]$. We get $h=g_1k_1$ as $u$ is not a zero divisor in $A_k[(h_{0,0}^{(i)})^{-1}]$. Therefore, $u$ is not a zero divisor in $A_k[(h_{0,0}^{(i)})^{-1}]/(g_1)\cong R_k$ as claimed. Since $R/(q_1)=\text{colim}R_k$, we conclude that $u$ is not a zero divisor in $R/(q_1)$.

Finally, suppose that $u^ny=0$ in $R$. Then, by the previous paragraph we get $u^{n-1}y=0$ in $R/(q_1)$. By induction on $n\geq 1$, we see that $y=0$ in $R/(q_1)$. This proves that $\text{Ann}(u^k)=(q_1)$ for all $k\geq 0$.
\end{proof}

We identify now the completion $R^{\wedge}_u$ and the localization $R[u^{-1}]$ with the lower left corner and the right upper coner of \eqref{eq:11.1} respectively.

\begin{lema}
    There is a map of $MU_\ast$-algebras 
    \begin{align*}
    \rho:R & \to MU_\ast[[u]]/[p]u \\
    u & \mapsto u \\
    d_{l,j}^{(i)} & \mapsto \sum_{k\geq 0} t_{l,j+k}^{(i)}u^k \\
    q_j & \mapsto \sum_{k\geq 0} c_{j+k}u^k \\
    \eta_i & \mapsto \sum_{k\geq 0} a_{0,1+k}^{(i)}u^k 
    \end{align*}
    
which induces an isomorphism $\widehat{\rho}:R^{\wedge}_u\to MU_\ast[[u]]/[p]u$.
\end{lema}

\begin{proof}
    We first prove that $\rho$ defined in this way satisfies the relations in \cref{teo:MU-calculation}. For \eqref{eq:relations-d}, we have
    \begin{align*}
        \rho(d_{l,j}^{(i)})-t_{l,j}^{(i)} & = \sum_{k\geq 0} t_{l,j+k}^{(i)}u^k - t_{l,j}^{(i)}\\
    & =t_{l,j}^{(i)}+\sum_{k\geq 1}t_{l,j+k}^{(i)}u^k-t_{l,j}^{(i)}\\
    & =\sum_{k\geq 1}t_{l,j+k}^{(i)}u^k\\
    & =u\sum_{k\geq 0}t_{l,j+1+k}^{(i)}u^k\\
    & =u\rho(d_{l,j+1}^{(i)}).
    \end{align*}
Relation \eqref{eq:relations-d-0} holds since $t_{0,j}^{(1)}=0$ for all $j\geq 0$. For \eqref{eq:relations-q}, we have
\[ \rho(q_{j})-c_{j}= \sum_{k\geq 0} c_{l,j+k}u^k - c_{j}=u\sum_{k\geq 0}c_{j+1+k}u^k=u\rho(q_{j+1}).\]
The relation $q_0=0$ follows since $\rho(q_0)=[p]u$ which is $0$ in $MU_\ast[[u]]/ [p]u$. Similarly, the relation $\eta_1=1$ follows since $a_{0,j}=0$ for $j\geq 2$ and $a_{0,1}=1$. Now we prove that relation \eqref{eq:relations-eta-d} holds. Since $\rho(\eta_i)$ is obtained by lowering by $1$ the power of $u$ in $[i]u$, we write
\[\rho(\eta_i)=\frac{[i]u}{u}.\]
From \eqref{eq:f_i}, we have the following in $MU_\ast[[u]]$
\[f_i(u)[i]u=[1+k_ip]u=u+_F [k_i]([p]u)=u+k_i [p]u+\cdots\]
The term $[p]u$ appears with power at least $2$ in all the higher terms of $[1+k_ip]u$. Thus, after lowering by $1$ the degree of $u$ in $[i]u$, we get
\[f_i(u)\frac{[i]u}{u}=1+k_i\frac{[p]u}{u}+\cdots\]
where $\frac{[p]u}{u}=\rho(q_1)$. Passing to the quotient $MU_\ast[[u]]/[p]u$, this expression becomes
\[f_i(u)\rho(\eta_i)=1+k_i\rho(q_1).\]
Now we identify $f_i(u)$ in terms of $d_{0,0}^{(i)}$. Since $([i]u)^{-1}=u^{-1}f_i(u)$ in $(MU[[u]]/[p]u)[u^{-1}]$, we get
\[([i]u)^{-1}=\sum_{k\geq-1}t_{0,k}^{(i)}u^k=u^{-1}\sum_{k\geq0}a^{(i^{-1})}_{0,k+1}([i]u)^k.\]
Over $MU_\ast[[u]]/[p]u$, we have
\[f_i(u)=\sum_{k\geq -1}t_{0,k}^{(i)}u^{k+1}= t_{0,-1}^{(i)}+u\sum_{k\geq0} t_{0,k}^{(i)}u^k=i^{-1}+u\rho(d_{0,0}^{(i)}).\]
This proves the desired relation. Finally, for \eqref{eq:relations-eta-q} note that $u\rho(q_1)=[p]u$, hence 
\begin{align*}
    \rho(\eta_i)\rho(q_1)& =(\sum_{k\geq 0}a_{0,1+k}^{(i)}u^k)\rho(q_1)=i\rho(q_1)+\sum_{k\geq 1}a_{0,k+1}^{(i)}u^k\rho(q_1)\\
    & =i\rho(q_1)+(\sum_{k\geq 0}a_{0,k+2}^{(i)}u^k)[p]u=i\rho(q_1).
\end{align*}
Now we show that $\rho$ induces an isomorphism $\widehat{\rho}:R^{\wedge}_u\to MU_\ast[[u]]/[p]u$. Since $MU_\ast[[u]]/[p]u$ is completed at $u$, the map $\rho$ induces a map 
\[\widehat{\rho}:R^{\wedge}_u\to MU_\ast[[u]]\to MU_\ast[[u]]/[p]u.\] 
We have a map $\widehat{\psi}:MU_\ast[[u]]\to R^{\wedge}_u$ induced by the map $\psi: MU_\ast[u]\to R$, given by $\psi(u)=u$. The composite $\widehat{\rho}\widehat{\psi}:MU_\ast[[u]]\to MU_\ast[[u]]/[p]u$ is the quotient map. To show that $\widehat{\psi}$ is onto, by induction on $n\geq 1$, relation \eqref{eq:relations-d} gives 
\[d_{l,j}^{(i)}=\sum_{k= 0}^{n-1}t_{l,j+k}^{(i)}u^k+u^nd_{l,j+n}^{(i)}.\]
The sequence $\{u^nd_{l,j+n}^{(i)}\}$ converges to $0$ in $R^{\wedge}_u$. So, we get
\[d_{l,j}^{(i)}=\sum_{k\geq0} t_{l,j+k}^{(i)}u^k\]
in $R^{\wedge}_u$. So, $d_{l,j}$ is in the image of $\widehat{\psi}$. A similar argument shows that $q_j$ is in the image of $\widehat{\psi}$ and that $q_j=\sum_{k\geq 0}c_{j+k} u^k$ in $R^{\wedge}_u$. In particular, it follows that $\widehat{\psi}([p]u)=q_0=0$, thus the map $\widehat{\psi}$ factors through a map
\[\widehat{\psi}:MU_\ast[[u]]/[p]u\to R^{\wedge}_u\]
It remains to show that $\eta_i$ is in the image of $\widehat{\psi}$. Rewriting relations \eqref{eq:relations-eta-d} and \eqref{eq:relations-eta-q} with what we just proved, we get in $R^{\wedge}_u$
\begin{equation*}
    \eta_i(i^{-1}+t_{0,0}^{(i)}u+\cdots) = 1+k_iq_1\tag{*}\label{eq:1}
\end{equation*}
\begin{equation*}
    \eta_i(p+c_{2}u+\cdots) = iq_1.\tag{**}\label{eq:2}
\end{equation*}
Multiplying \eqref{eq:1} by $i$ and \eqref{eq:2} by $k_i$ and subtracting the resulting expressions, we get
\[\eta_i(1+(it_{0,0}^{(i)}-k_ic_2)u+\cdots)=i\]
The power series $1+(it_{0,0}^{(i)}-k_ic_2)u+\cdots$ is a unit in $MU_\ast[[u]]$, so it is also a unit in $R^{\wedge}_u$. Therefore, we obtain
\[\eta_i=i(1+(it_{0,0}^{(i)}-k_ic_2)u+\cdots)^{-1}\]
and the right hand side is clearly in the image of $\widehat{\psi}$.
\end{proof}

\begin{lema}\label{lema:localization-R}
There is a map of $MU_\ast$-algebras
\begin{align*}
    \kappa: R & \to MU_\ast[u_i^{\pm 1},d_l^{(i)}] \\
u & \mapsto u \\
d_{l,j}^{(i)} & \mapsto u^{-j}d_{l}^{(i)}-\sum_{k=1}^{l+1+j}t_{l,j-k}^{(i)}u^{-k} \\
q_j & \mapsto -\sum_{k=1}^{j}c_{j-k}u^{-k} \\
\eta_i & \mapsto u^{-1}u_i 
\end{align*}
which induces an isomorphism $\overline{\kappa}:R[u^{-1}]\to MU_\ast[u_i^{\pm 1},d_l^{(i)}]$.
\end{lema}
\begin{proof}
    Again, we start by proving that $\kappa$ defined in this way satisfies the relations in $R$. For relation \eqref{eq:relations-d}, we have
    \begin{align*}
          u\kappa(d_{l,j+1}^{(i)}) & =u(u^{-j-1}d_l^{(i)}-\sum_{k=1}^{l+1+j+1}t_{l,j+1-k}^{(i)}u^{-k})\\
          & =u^{-j}d_l^{(i)}-t_{l,j}^{(i)}-\sum_{k=2}^{l+1+j+1}t_{l,j+1-k}^{(i)}u^{1-k}\\
          & =u^{-j}d_l^{(i)}-\sum_{s=1}^{l+1+j}t_{l,j-s}^{(i)}u^{-s}-t_{l,j}^{(i)}\\
          & =\kappa(d_{l,j}^{(i)})-t_{l,j}^{(i)}.
    \end{align*}
Note that $\kappa(d_{0,j}^{(1)})=u^{-j}d_0^{(1)}-u^{-1-j}$ because $t_{0,j}^{(1)}=0$ for $j\geq 0$. Since $d_0^{(1)}=u_1^{-1}=u^{-1}$, we get $\kappa(d_{0,j}^{(1)})=0$. Similarly, for \eqref{eq:relations-q} we get
\[u\kappa(q_{j+1})=-u\sum_{k=1}^{j+1}c_{j+1-k}u^{-k}=-c_j - \sum_{k=2}^{j+1}c_{j+1-k}u^{1-k}=-c_j-\sum_{s=1}^j c_{j-s}u^{-s}=\kappa(q_j)-c_j.\]
The relation $q_0=0$ follows by convention, and $\eta_1=1$ clearly holds as $u_1=u$. For \eqref{eq:relations-eta-d}, note that $\kappa(q_1)=0$ since $c_0=0$. On the other hand, we have
\[\kappa(\eta_i)(u\kappa(d_{0,0}^{(i)})+i^{-1})=u^{-1}u_i(u(d_0^{(i)}-t_{0,-1}^{(i)}u^{-1})+i^{-1})\]
\[=u^{-1}u_i(u(u_i^{-1}-i^{-1}u^{-1})+i^{-1})=1=1+k_i\kappa(q_1).\]
Relation \eqref{eq:relations-eta-q} follows because $\kappa(q_1)=0$. 

It remains to prove that $\kappa$ induces an isomorphism $\overline{\kappa}:R[u^{-1}]\to MU_\ast[u_i^{\pm 1}, d_l^{(i)}]$. Since $u=u_1$ is a unit in $MU_\ast[u_i^{\pm 1}, d_l^{(i)}]$, from the universal property of localizations we get a map 
\begin{align*}
    \overline{\kappa}:R[u^{-1}] & \to MU_\ast[u_i^{\pm 1}, d_l^{(i)}]
\end{align*}
By Lemma \ref{lema:bounded-torsion}, every $q_1$-divisible element in $R$ becomes zero in $R[u^{-1}]$. In particular, relation \eqref{eq:relations-eta-d} gives that $\eta_i$ is a unit in $R[u^{-1}]$. With this in mind we define
\begin{align*}
    \overline{\kappa}^{-1}:MU_\ast[u_i^{\pm 1},d_l^{(i)}] & \to R[u^{-1}] \\
    u_i & \mapsto u\eta_i \\
    u_i^{-1} & \mapsto d_{0,0}^{(i)}+u^{-1}i^{-1} \\
    d_l^{(i)}& \mapsto d_{l,0}^{(i)}+\sum_{k=1}^{l+1}t_{l,-k}^{(i)}u^{-k}.
\end{align*}
This map is consistent with $u=u_1$ as $\eta_1=1$. Similarly, it is consistent with the fact that $d_0^{(i)}=u^{-1}$ since $\overline{\kappa}^{-1}(d_0^{(i)})=d_{0,0}^{(i)}+u^{-1}t_{0,-1}^{(i)}=d_{0,0}^{(i)}+u^{-1}i^{-1}$. On the other hand, $\overline{\kappa}^{-1}(u_i^{-1})=u^{-1}\eta_i^{-1}$ but from \eqref{eq:relations-eta-d}, we see that $\eta_i^{-1}=ud_{0,0}^{(i)}+i^{-1}$. It is easy to see that $\overline{\kappa}\circ \overline{\kappa}^{-1}$ is the identity. For the the other composition, it is straightforward to see that $\overline{\kappa}^{-1}(\overline{\kappa}(\eta_i))=\eta_i$, and $\overline{\kappa}^{-1}(\overline{\kappa}(q_j))=q_j$, using the fact that in $R[u^{-1}]$ we can write relations \eqref{eq:relations-q} as $q_{j+1}=u^{-1}(q_j-c_j)$. Finally,
\[\overline{\kappa}^{-1}(\overline{\kappa}(d_{l,j}^{(i)}))=u^{-j}(d_{l,0}^{(i)}+\sum_{k=1}^{l+1}t_{l,-k}u^{-k})-\sum_{k=1}^{j+1+l}t_{l,j-k}^{(i)}u^{-k}.\]
Using \eqref{eq:relations-d}, we have $d_{l,0}^{(i)}=\sum_{k= 0}^{j-1}t_{l,k}^{(i)}u^k+u^j d_{l,j}^{(i)}$. Therefore,
\[\overline{\kappa}^{-1}(\overline{\kappa}(d_{l,j}^{(i)}))=d_{l,j}^{(i)}+\sum_{k= 0}^{j-1}t_{l,k}^{(i)}u^{k-j}+\sum_{k=1}^{l+1}t_{l,-k}u^{-k-j}-\sum_{k=1}^{j+1+l}t_{l,j-k}^{(i)}u^{-k}\]
\[=d_{l,j}^{(i)}+\sum_{k= 0}^{j-1}t_{l,k}^{(i)}u^{k-j}+\sum_{k=1}^{l+1}t_{l,-k}u^{-j-k}- \sum_{k=1}^{j}t_{l,j-k}u^{-k}-\sum_{k=1}^{l+1}t_{l,-k}u^{-j-k}\]
\[=d_{l,j}^{(i)}.\]
\end{proof}

\begin{proof}[Proof of Theorem \ref{teo:MU-calculation}]
    All together, the previous lemmas state precisely that the right side and lower left corner of Proposition \ref{prop:pullback-strickland} are isomorphic in a compatible way (that is, all the resulting diagrams commute) to the respective terms in \eqref{eq:11.1}. Thus, we get a natural isomorphism $R\to MU^{C_p}_\ast$.
\end{proof}
We also record the following lemma.
\begin{lema}\label{lema:congruent}
    We have $\eta_i\equiv i$ mod $u$ in $R$. In particular, $R/(u)\cong MU_\ast$.
\end{lema}

\begin{proof}
   From \eqref{eq:relations-q} we get $p=q_1-uq_2$. Then, multiplying relation \eqref{eq:relations-eta-d} by $i$ we get
   \begin{align*}
       \eta_i(ud_{0,0}^{(i)}i+1+k_ip)& =i+k_iiq_1\\
       \eta_i(ud_{0,0}^{(i)}i+1+k_i(q_1-uq_2)) & =i+k_iiq_1\\
       \eta_i+k_i(\eta_i q_1-iq_1)-i & =u(\eta_i k_i q_2-\eta_i d_{0,0}^{(i)}i)
   \end{align*}
   The term $\eta_i q_1-iq_1$ is zero by \eqref{eq:relations-eta-q}, so we get
   \begin{align*}
       \eta_i-i =u(\eta_i k_i q_2-\eta_i d_{0,0}^{(i)}i).
   \end{align*}
    The second part of the statement follows from the fact that relations \eqref{eq:relations-d} and \eqref{eq:relations-q} imply that the generators $d_{l,j}^{(i)}$ and $q_j$ are congruent to some element in $MU_\ast$ mod $u$.
\end{proof}

\begin{remark}\label{remark:kernel}
    From Lemma \ref{lema:bounded-torsion}, we have that $\ker(\kappa)=(q_1)$. Moreover, from Lemma \ref{lema:congruent} every element $x\in MU^{C_p}_\ast$ is congruent to an of $MU_\ast$ modulo $u$. Therefore $(q_1)\subset MU^{C_p}_\ast$ is the $MU_\ast$-module generated by $q_1$.
\end{remark}

\vspace{0.5cm}

\section{Calculation of \texorpdfstring{$\Omega^{C_p}_\ast$}{Omega Cp *}}\label{section4}
In this section, we prove the main theorem of this paper, which describes the $MU_\ast$-algebra $\Omega^{C_p}_\ast$ in terms of its generators and relations. In the pullback diagram \eqref{eq:09}, the lower left corner now has the presentation as given by Theorem \ref{teo:MU-calculation}. The lower horizontal map in this diagram, which is localization at $u$, is described by $\kappa$ in Lemma \ref{lema:localization-R}.

\begin{teo}\label{thm:generators-geometric-cobordism}
    As an $MU_\ast$-algebra, $\Omega^{C_p}_\ast$ has generators
    \[d_{l,j}^{(i)}, q_j \text{, with } |d_{l,j}^{(i)}|=2(l+j+1), |q_j|=2(j-1)\]
which are subject of the following relations
    \begin{align*}
        d^{(i)}_{l,j+1}(d^{(i')}_{k,s}-t^{(i')}_{k,s}) & =d^{(i')}_{k,s+1}(d^{(i)}_{l,j}-t^{(i)}_{l,j})\tag{20}\label{eq:19} \\
    d^{(i)}_{l,j+1}(q_k-c_k) & =q_{k+1}(d^{(i)}_{l,j}-t^{(i)}_{l,j}) \tag{21}\label{eq:20}\\
    q_{j+1}(q_k-c_k) & =q_{k+1}(q_j-c_j) \tag{22}\label{eq:21}\\
    q_0 & =0 \tag{23}\label{eq:22}\\
    d_{0,j}^{(1)} & =0 \text{ for all } j\geq 0 \tag{24}\label{eq:23}.
    \end{align*}
\end{teo}

\begin{remark}\label{remark:4.2}
    As with Theorem \ref{teo:MU-calculation}, we can omit the generators $d_{0,j}^{(i)}$ for $j\geq 0$, and $q_0$, by adding the relations 
    \begin{align*}\tag{25}\label{eq:24}
        q_1(q_k-c_k) & =0
    \end{align*}
    \begin{align*}\tag{26}\label{25}
        q_1(d_{l,j}^{(i)}-t_{l,j}^{(i)}) & =0
    \end{align*} 
\end{remark}

The proof of this theorem will follow from the next lemmas. 
\begin{lema}\label{lema:generated}
    $\Omega^{C_p}_\ast$ is the $MU_\ast$-subalgebra of $MU^{C_p}_\ast$ generated by $q_j$ and $d_{l,j}^{(i)}$.
\end{lema}

\begin{proof}
Let $K_0$ be the subalgebra of $MU_\ast^{C_p}$ generated by $q_j$, and $d_{l,j}^{(i)}$. Since $\Omega^{C_p}_\ast$ is given by a pullback along an inclusion, we have 
\[\Omega^{C_p}_\ast=\kappa^{-1}(MU_\ast[u_i^{- 1},d_l^{(i)}]).\]
Where $\kappa: MU_\ast^{C_p}\to MU_\ast[u_i^{\pm 1},d_l^{(i)}]$ is the map given in Lemma \ref{lema:localization-R}. From the definition of $\kappa$, it is clear that $K_0\subset \Omega^{C_p}_\ast$. Suppose $x\in \Omega^{C_p}$. We will show that $x\in K_0$ by constructing a polynomial $f\in K_0$ such that $\kappa(x)=\kappa(f)$. Consider the lexicographic order on the monomials of $MU_\ast[u_i^{- 1},d_l^{(i)}]$ given by
\[u^{-1}< u_2^{-1}< \cdots < u_{p-1}^{-1}< d_l^{(i)}\]
\[d_l^{(i)}< d_{s}^{(i')} \text{ if $l<s$, or if $l=s$ and $i<i'$ }.\]
With this ordering, we have
\begin{align*}
    \kappa(d_{l,j}^{(i)}) & = u^{-j}d_{l}^{(i)}+\text{LOT}\\
    \kappa(d_{0,j}^{(i)}) & = u^{-j}u^{-1}_i+\text{LOT} \\
    \kappa(q_j) & = -pu^{-j+1}+\text{LOT}
\end{align*}
where LOT means lower order terms. Starting with $x=x_0$, suppose
\[\kappa(x)= c u^{-N}u_2^{-n_2}\cdots u_{p-1}^{-n_{p-1}}(d_{l_{1}}^{(i_1)})^{m_1}\cdots (d_{l_{k}}^{(i_k)})^{m_k}+ \text{LOT}\]
where $c\in MU_\ast$, and $d_{l_1}^{(i_1)}<\cdots <d_{l_k}^{(i_k)}$. We can find $f_0\in K_0$ such that $\text{LM}(\kappa(f_0))=\text{LM}(\kappa(x))$, where LM means leading monomial. For instance, we could take
\[f_0=c (d_{0,0}^{(2)})^{n_2}\cdots (d_{0,0}^{(p-1)})^{-n_{p-1}}(d_{l_{i_1},0}^{(i_1)})^{m_1}\cdots (d_{l_{i_{k-1}},0}^{(i_{k-1})})^{m_{i_{k-1}}}(d_{l_{i_k},0}^{(i_k)})^{m_k-1}(d_{l_{i_k},N}^{(i_k)}).\]
Set $x_1=x-f_0$. Then $\text{LM}(\kappa(x_1))<\text{LM}(\kappa(x))$. Note that as long as there are non-zero powers of $u^{-1}_2,\ldots, u_{p-1}^{-1}$  or $d_l^{(i)}$, we can apply the same argument to find $f_r\in K_0$ for $r\geq 1$ such that if $x_r=x_{r-1}-f_{r-1}$, then $\text{LM}(\kappa(f_r))=\text{LM}(\kappa(x_r))$. Since $\text{LM}(\kappa(x_r-f_r))< \text{LM}(\kappa(x_r))$ at each step $r$, there must be some $R$ such that 
\[\kappa(x_R)=c_Ru^{-n_R}+\text{LOT}\]
with $c_R\in MU_\ast$, $n_R\in \N$, and $\text{LOT}(\kappa(x_R))\in MU_\ast[u^{-1}]$. By \cite{kosniowski1976generators}*{Lemma 3.4}, $c_R=pc'_R$ is divisible by $p$ in $MU_\ast$. Setting $f_{R}=-c'_Rq_{n_R+1}\in K_0$, we get $\text{LM}(\kappa(f_{R}))=\text{LM}(\kappa(x_R))$. Since $\text{LOT}(\kappa(q_{j}))\in MU_\ast[u^{-1}]$, we can go on to define $x_{R+1}=x_R-f_R$ without adding new monomials containing $u^{-1}_2,\ldots,u_{p-1}^{-1}$,or $ d_l^{(i)}$, and the process can be continued. Thus, there is a $T\in \N$ such that $f=f_0+\cdots+f_R+f_{R+1}+\cdots+f_{R+T}\in K_0$ satisfies
\[\kappa(x-f)=0.\]
Since $\ker(\kappa)=(q_1)\subset MU^{C_p}_\ast$, there must be an $A\in MU^{C_p}_\ast$ with $Aq_1=x-f$. By Remark \ref{remark:kernel}, $Aq_1=aq_1$ for some $a\in MU_\ast$, so $x=f+aq_1\in K_0$.
\end{proof}

\begin{proof}[Proof of Theorem \ref{thm:generators-geometric-cobordism}]
    Let 
    \[R=MU_\ast[q_k,d_{l,j}^{(i)}\mid k\geq 1, (i,l)\neq (1,0)]/I\]
    where $I$ is the ideal generated by the relations \eqref{eq:19}, \eqref{eq:20}, \eqref{eq:21}, \eqref{eq:24}, and \eqref{25}. From the relations \eqref{eq:relations-d} and \eqref{eq:relations-q} in $MU^{C_p}_\ast$, we get
    \begin{align*}
         \tag{27}\label{last}ud_{l,j+1}^{(i)}d_{s,k+1}^{(i')} & =d^{(i)}_{l,j+1}(d^{(i')}_{k,s}-t^{(i')}_{k,s}) =d^{(i')}_{k,s+1}(d^{(i)}_{l,j}-t^{(i)}_{l,j}) \\
        ud_{l,j+1}^{(i)}q_{k+1}&= d^{(i)}_{l,j+1}(q_k-c_k) =q_{k+1}(d^{(i)}_{l,j}-t^{(i)}_{l,j})\\
        uq_{j+1}q_{k+1} & =q_{j+1}(q_k-c_k) =q_{k+1}(q_j-c_j).    
    \end{align*}
Hence, $\Omega^{C_p}_\ast$ satisfies the relations in $I$. We get a well-defined surjective $MU_\ast$-algebra map
\begin{align*}
    \Lambda :R & \to \Omega^{C_p}_\ast \\
    q_k & \mapsto q_k \\
    d_{l,j}^{(i)} & \mapsto d_{l,j}^{(i)}
\end{align*}
Consider the following commutative diagram
\[\begin{tikzcd}
0 \arrow[r] & \ker(\kappa|_{\Omega^{C_p}_\ast}\circ \Lambda) \arrow[d] \arrow[r] & R \arrow[d, "\Lambda"] \arrow[r, "\kappa|_{\Omega^{C_p}_\ast}\circ \Lambda"] & {MU_\ast[u_i^{\pm 1}, d_l^{(i)}]} \arrow[d, "id"] \\
0 \arrow[r] & \ker(\kappa|_{\Omega^{C_p}_\ast}) \arrow[r]                     & \Omega^{C_p}_\ast \arrow[r, "\kappa|_{\Omega^{C_p}_\ast}"]             & {MU_\ast[u_i^{\pm 1}, d_l^{(i)}]}                
\end{tikzcd}\]
where $\kappa|_{\Omega^{C_p}_\ast}$ is the restriction of the map $\kappa$ in Lemma \ref{lema:localization-R}. We will show that
\[\ker(\kappa|_{\Omega^{C_p}_\ast}\circ \Lambda)= \ker(\kappa|_{\Omega^{C_p}_\ast})\tag{28}\label{eq:kernel}=MU_\ast\{q_1\}.\]
By Remark \ref{remark:kernel}, $\ker(\kappa|_{\Omega^{C_p}_\ast})=\ker(\kappa)$ is the $MU_\ast$-module generated by $q_1$. If \eqref{eq:kernel} holds, by the Four Lemma, the surjective map $\Lambda$ is also injective, and $R\cong \Omega^{C_p}_\ast$ as $MU_\ast$-algebras. 

To show \eqref{eq:kernel}, we start by ordering the generators of $R$ as follows
\begin{align*}
    &q_1<q_2<q_3<\cdots \\
     < & d_{0,0}^{(2)}<d_{0,1}^{(2)}<d_{0,2}^{(2)}<\cdots <d_{1,0}^{(2)}<d_{1,1}^{(2)}<\cdots\\
     < & d_{0,0}^{(3)}<d_{0,1}^{(3)}<\cdots \\
     < & \cdots \\
     < & d_{0,0}^{(p-1)} < d_{0,1}^{(p-1)}<\cdots \\
     < & d_{1,0}^{(1)}<d_{1,1}^{(1)}<\cdots 
\end{align*}
Given this ordering, we can rename the generators $q_k$ and $d_{l,j}^{(i)}$ as $x_k$ for $k\leq \omega^2$ where $\omega$ is the first infinite ordinal. We have
\begin{align*}
    x_k & =q_{k+1} \text{ for $k\in \N$} \\
    x_{(l(p-1)+i-1)\omega + j} & =d_{l,j}^{(i)}.
\end{align*}
In the same way, define 
\begin{align*}
    p_k & =c_{k+1} \in MU_\ast \\
    p_{(l(p-1)+i-1)\omega + j} & =t_{l,j}^{(i)}\in MU_\ast.
\end{align*}
The relations \eqref{last} can be rewritten as
\begin{align*}
    x_sx_t=x_{s-1}x_{t+1}+p_tx_s-p_{s-1}x_{t+1}\tag{29}\label{eq:x-relations}
\end{align*}
for all $s,t<\omega^2$ and $s\neq n\omega $. In particular, for $q_1$ we have
\[q_1x_t=p_tq_1\tag{30}\label{eq:new-q}\]
for all $t<\omega^2$. Consider a monomial 
\[A= x_{k_1}x_{k_2}\cdots x_{k_N}\in R\]
where $k_1\leq k_2\leq \cdots \leq k_N$. We say that $A$ has length $N$. Suppose that $N\geq 2$. If for some $1\leq s\leq N-1$, $k_s=n\omega+j$ for some $j\geq 1$, we can use \eqref{eq:x-relations} to rewrite
\[A=x_{k_1}\cdots x_{k_{s}-1}\cdots x_{k_{N}+1}+ \text{LLT}\]
where LLT means lower length terms. Repeating this process, we can write
\[A=q_1^md_{l_{m+1},0}^{(i_{m+1})}\cdots d_{l_{N-1},0}^{(i_{N-1})}d_{l_N,j}^{(i_N)} +\text{LLT}.\]
In particular, we have 
\[(l_{m+1},i_{m+1})\leq \cdots \leq(l_N,i_N)\]
in the lexicographic ordering of pairs. If $m\geq 1$, we use \eqref{eq:new-q} to get
\[A=cq_1+\text{LLT}\]
for some $c\in MU_\ast$. 

For $I=\{(l_1,i_1)\leq \ldots\ \leq (l_{N},i_{N})\}$ and $j\geq 0$, denote 
\[d_{I,j}=d_{l_1,0}^{(i_1)}\cdots d_{l_{N-1},0}^{(i_{N-1})}d_{l_N,j}^{(i_N)}\]
Repeating the process for the lower length terms of $A$, we can write
\[A=\sum c_sd_{I_s,j_s}+\sum r_tq_{k_t}.\]
with $c_s,r_t\in MU_\ast$. In other words, the monomials
$d_{I,j}$ , $q_k$ generate $R$ as an $MU_\ast$-module. We give a new ordering to these elements by
\begin{align*}
    q_k & <q_{k'} \text{ if $k<k'$} \\
    q_k & < d_{I,j} \text{ for all $k, I, j$}\\
    d_{I,j} & < d_{I',j'} \text{ if $I<I'$}\\
    d_{I,j} & < d_{I,j'} \text{ if $j<j'$}
\end{align*}
We have 
\begin{align*}
    (\kappa\circ\Lambda)(d_{I,j}) & = u^{-j}d_{l_1}^{(i_1)}\cdots d_{l_n}^{(i_n)}+\text{LOT} \\
    (\kappa\circ\Lambda)(q_k) & = -pu^{-k+1}+\text{LOT}
\end{align*}
where the monomial order in $MU_\ast[u_i^{- 1},d_l^{(i)}]$ is the same as in Lemma \ref{lema:generated}.
In particular, if $d_{I,j}<d_{I',j'}$, then $\text{LM} ((\kappa\circ\phi)(d_{I,j}))< \text{LM}( (\kappa\circ\phi)(d_{I',j'}))$, where LM means leading monomial. Similarly, if $q_k<q_{k'}$ then $\text{LM}((\kappa\circ\phi)(q_k)< \text{LM}( (\kappa\circ\phi)(q_{k'})$. Therefore, if $\text{LM}(f)\neq q_1$ it must be $(\kappa\circ\phi)((f))\neq 0$. This shows that $\ker(\kappa\circ \phi)$ is the $MU_\ast$-module generated by $q_1$. 
\end{proof}

In \cite{kosniowski1976generators}*{Theorem 4.1}, Kosniowski gave a basis for $\Omega^{C_p}_\ast$ as a free $MU_\ast$-module. The proof of Theorem \ref{thm:generators-geometric-cobordism} gives a new basis.

\begin{cor}\label{corolario}
    As a free $MU_\ast$-module, $\Omega^{C_p}_\ast=MU_\ast\{q_k,d_{I,j}\}$ where $k\geq 1, j\geq 0$, and $I$ ranges over all sequences of pairs $\{(l_1,i_1)\leq \cdots\leq (l_N,i_N)\}$ with $N\geq 1$, $l_k\geq 0$, $1\leq i_k\leq p-1$, and $(l_k,i_k)\neq (0,1)$.
\end{cor}

\begin{proof}
    The proof of Theorem \ref{thm:generators-geometric-cobordism} shows that $\Omega^{C_p}_\ast$ is generated as an $MU_\ast$-module by these elements. Moreover, if
    \[\sum_{s=1}^n c_sd_{I_s,j_s}+\sum_{t=1}^m r_tq_{k_t} + wq_1=0\]
    with $c_s,r_t,w\in MU_\ast$, and $k_t\neq 1$; the same proof shows that $c_s=r_t=0$ for all $1\leq s\leq n$, $1\leq t\leq m$. Hence, it is enough to show that $q_1$ has no $MU_\ast$-torsion. Suppose $wq_1=0$ for some $w\in MU_\ast$. Since $q_1\in \ker(\kappa)$, it represents a $C_p$-manifold with free action. Moreover, since $q_1$ has degree $0$ and satisfies $q_1^2=pq_1$, it must contain $p$ points. So it is a single free orbit, that is, $q_1=C_p$. Consider the restriction map
\begin{align*}
    \res:\Omega^{C_p}_\ast & \to MU_\ast 
\end{align*}
which is given by forgetting the action on a manifold. Therefore, $\res(wq_1)=w\res(q_1)=wp=0$. Since $MU_\ast$ is a domain, we must have $w=0$. 
\end{proof}

\vspace{0.5cm}

\section{Geometric Generators for \texorpdfstring{$\Omega^{C_p}_\ast$}{Omega Cp *}}\label{section5}

In \cite{kosniowski1976generators}, Kosniowski gave a set of geometric $MU_\ast$-algebra generators for $\Omega^{C_p}_\ast$, though without relations among them. In this section, we relate Kosniowski's generators with those of Theorem \ref{thm:generators-geometric-cobordism}. 

If $V$ is a complex $C_p$-representation, then $\C P(V)=(V-\{0\})/\C^\times$ is the space of all $1$-dimensional complex subspaces in $V$. We will write 
\[\C P(n_0,n_1,\ldots , n_{p-1})=\C P(\C_0^{n_0}\oplus \C_1^{n_1}\oplus\cdots\oplus \C_{p-1}^{n_{p-1}}).\]
We will omit the index $n_k$ if $n_k=0$. Here, $\C_k$ is the $1$-dimensional complex $C_p$-representation given by the character $\zeta^k=e^{\frac{2\pi ik}{p}}$.  In particular, $\C P(n_0)=\C P^{n_0-1}$.

 We recall the following construction of \cite{kosniowski1976generators}. Let $M$ be a stably almost complex $C_p$-manifold such that the $C_p$-action extends to an $S^1$-action. Then $S^1$ acts freely on $M\times S^3$ by $\gamma\cdot (m,z_1,z_2)=(\gamma \cdot m,\gamma z_1,\gamma z_2)$. Here, $S^1\subset \C$ and $S^3\subset \C^2$ are the unit spheres, and $(z_1,z_2)\in S^3$. In \cite{kosniowski1976generators}, Kosniowski defined 
 \[\Gamma(M)=(M\times S^3)/S^1.\] 
 This is a stably almost complex $C_p$-manifold with action given by $\gamma\cdot [m,z_1,z_2]=[\gamma\cdot m,\gamma z_1,z_2]$. This construction can be iterated, with $\Gamma^n(M)=\Gamma(\Gamma^{n-1} M)$, and $\Gamma^0(M)=M$.

 The following is the main theorem of \cite{kosniowski1976generators}.

\begin{teo}\cite{kosniowski1976generators}*{Theorem 2.8}\label{teo-geo-generators}
    $\Omega^{C_p,\text{geo}}_\ast$ is generated as an $MU_\ast$-algebra by the bordism classes of the following manifolds:
    \begin{enumerate}[(i)]
        \item $C_p$
        \item $\Gamma^m(pt)$ with $m\geq 0.$
        \item $\Gamma^m(\C P(1_0,1_i))$ with $m\geq 0$ and $\frac{p+1}{2}\leq i<p-1$.
        \item $\Gamma^m(\C P(1_0,1_1,1_i))$ with $m\geq 0$ and $1< i\leq \frac{p+1}{2}$.
        \item $\Gamma^m(\C P(n_0,1_i))$ with $m\geq 0$, $1\leq  i\leq p-1$, and $n\geq 2$.
        \item $S_i$ with $1\leq i\leq \frac{p-1}{2}$ where $S_i$ is the Riemann surface of genus $\frac{(p-i^{-1}-1)(p-1)}{2}$ associated to the function 
        \[u^{p-i^{-1}}-z^p=-1\]
        where $i^{-1}$ is defined as in Section \ref{section:5}.
   \end{enumerate}
\end{teo}

\begin{remark}
    Note that $\Gamma(pt)=\C P(1_0,1_{p-1})$. Also, $S_i$ is the projective closure of the algebraic curve given by $u^{p-i^{-1}}-z^p=-1$. The action of $C_p$ on $S_i$ is induced by $\zeta\cdot[z,u]=[\zeta z,u]$ where we see the primitive $p$-root of unity $\zeta=e^{\frac{2\pi i}{p}}$ as a generator of $C_p$. 
\end{remark}

We will relate our generators to these manifolds by comparing their image under the geometric fixed points map. Under the isomorphism in Lemma \ref{lema:identify-geo-fixed-points}, the map $\Phi^{\text{geo}}:\Omega^{C_p,\text{geo}}_\ast\to \Phi^{C_p}\Omega^{C_p,\text{geo}}_\ast$ is given by
\begin{align*}
  \Phi^{\text{geo}}:  \Omega^{C_p,\text{geo}}_\ast & \to \bigoplus_{m_1,\ldots,m_{n}\in \N_{\geq 0}}MU_{\ast-2(m_1+\ldots+m_{n})}( BU(m_1)\times\cdots\times BU(m_n)) \\
   [M] & \mapsto \bigoplus_k [M^{C_p}_k\to BU(m_1)\times\cdots\times BU(m_n)]
\end{align*}
where $M^{C_p}_k$ is a connected component of $M^{C_p}$, and $M^{C_p}_k\to BU(m_1)\times\cdots\times BU(m_n)$ classifies the normal bundle of the inclusion $M^{C_p}_k\hookrightarrow M$. Here, the $i$-th component map $M^{C_p}_k\to BU(m_i)$ classifies the $\C_i$-isotypical component of this normal bundle. 

In the presentation $\Phi^{C_p}\Omega^{C_p}_\ast=MU_\ast[u^{-1}_i,u^{-1}_ib_l^{(i)\prime}]$, the normal bundle of the inclusion 
\[\C P(n_0)\to \C P(n_0,1_i)\]
can be identified with $u^{-1}_ib_{n_0-1}^{(i)\prime}$ (see \cite{hanke2005geometric} or \cite{conner1967bordism}*{Lemma 2.2}). However, in the presentation $\Phi^{C_p}\Omega^{C_p}_\ast=MU_\ast[u_i^{-1},d_l^{(i)}]$ the map $\kappa: \Omega^{C_p}\to MU_\ast[u_i^{-1},d_l^{(i)}]$ is $\iota\circ \Phi^{\Omega^{C_p}}$ (see \eqref{eq:09}). Hence, in this presentation, the normal bundle of the inclusion
\[\C P(n_0)\to \C P(n_0,1_i)\]
can be identified with the generator $d_{n_0-1}^{(i)}=\iota(u^{-1}_ib_{n_0-1}^{(i)\prime})$. Using the top diagram in \eqref{diagram:big-square}, we can identify the map $\Phi^{\text{geo}}:\Omega^{C_p,\text{geo}}_\ast\to MU_\ast[u_i^{-1},d_l^{(i)}]$ with the map $\kappa$.

The following lemma is the main step used in \cite{kosniowski1976generators}.

\begin{lema}\cite{kosniowski1976generators}*{Lemma 3.1} \label{lema:kos}
\hfill
    \begin{enumerate}[(i)]
        \item $\Phi^{\text{geo}} (C_p)=0$.
        \item $\Phi^{\text{geo}}(\C P(1_0,1_i))=u_{p-i}^{-1}+u_i^{-1}$.
    \item $\Phi^{\text{geo}}(\C P(1_0,1_1,1_i))=u^{-1}u_i^{-1}+u_{p-1}^{-1}u^{-1}_{i-1}+u^{-1}_{p-i}u^{-1}_{p-i+1}$.
    \item $\Phi^{\text{geo}}(\C P(n_0,1_i))=d_{n-1}^{(i)}+u^{-n}_{p-i}$
    \item $\Phi^{\text{geo}}(S_i)=u_i^{-1}+(p-i^{-1})u^{-1}$.
    \item $\Gamma(M)=u^{-1}\Phi^{\text{geo}}(M)+ \res(M) u^{-1}_{p-1}$ where $\res:\Omega^{C_p}_\ast\to MU_\ast$ is the restriction to the trivial subgroup which forgets the $C_p$-action. 
    \end{enumerate} 
\end{lema}

The following lemma gives some of Kosniowski’s geometrically-defined generators in terms of our generators.

\begin{lema}\label{lemma:last}
We have the following identifications.
\begin{enumerate}[(i)]
    \item $C_p=q_1$.
    \item $S_i=d_{0,0}^{(i)}-q_2$ for $1\leq i\leq \frac{p-1}{2}$ mod $q_1$.
    \item $\C P(1_0,1_i)=d_{0,0}^{(p-i)}+d_{0,0}^{(i)}-q_2$ for $\frac{p+1}{2}\leq i\leq p-1$ mod $q_1$.
    \item For $1<i\leq \frac{p-1}{2}$, 
    \[\C P(1_0,1_1,1_i)=d_{0,0}^{(p-1)}d_{0,0}^{(i-1)}+(i-1)^{-1}d_{0,1}^{(p-1)}+d_{0,0}^{(p-i+1)}d_{0,0}^{(p-i)}+(p-i)^{-1}d_{0,1}^{(p-i+1)}+(p-i+1)^{-1}d_{0,1}^{(p-i)}\]
    \[+d_{0,1}^{(i)}+(p-1)d_{0,1}^{(i-1)}-N_{p,i}q_3-M_{p,i}q_2\]
    mod $q_1$. Here, the terms $N_{p,i}\in \N$, and $M_{p,i}\in MU_\ast$ satisfy
\begin{align*}
    pN_{p,i} & =(p-i+1)^{-1}(p-i)^{-1}+i^{-1}+(p-1)(i-1)^{-1}\tag{31}\label{30}
\end{align*}
\begin{align*}
    p M_{p,i} & = (p-i)^{-1}t_{0,0}^{(p-i+1)}+(i-1)^{-1}t_{0,0}^{(p-1)}+(p-i+1)^{-1}t_{0,0}^{(p-i)}\tag{32}\label{31}\\
    & +t_{0,0}^{(i)}+(p-1)t_{0,0}^{(i-1)}-N_{p,i}c_2.
\end{align*}
\end{enumerate}
\end{lema}

\begin{proof}
    \begin{enumerate}[(i)]
        \item This was proved in Corollary \ref{corolario}.
        
        \item Since $\kappa(q_2)=-pu^{-1}$ and $\Phi^{\text{geo}}S_1=pu^{-1}$, we conclude $q_2=-S_1$ mod $q_1$. For $1<i\leq \frac{p-1}{2}$,  We have
\[\Phi^{\text{geo}}S_i=u^{-1}_i+(p-i^{-1})u^{-1}=u^{-1}_i-i^{-1}u^{-1}+pu^{-1}=\kappa(d_{0,0}^{(i)}-q_2).\]

\item Using that $p-i=-i$ mod $p$ and that $(p-i)(p-i)^{-1}=1$ mod $p$, we can conclude $i^{-1}+(p-i)^{-1}=p$ since $1\leq i^{-1}, (p-i)^{-1}\leq p-1$. Hence,
\begin{align*}
    \Phi^{\text{geo}}\C P(1_0,1_i) & =u^{-1}_{p-i}+u^{-1}_i=u^{-1}_{p-i}+u^{-1}_i-pu^{-1}+pu^{-1}\\
    &=u^{-1}_{p-i}-(p-i)^{-1}u^{-1}+u^{-1}_i-i^{-1}u^{-1}+pu^{-1}\\
    &= \kappa(d_{0,0}^{(p-i)}+d_{0,0}^{(i)}-q_2).
\end{align*}
\item The reader can check that both sides have the same image under the geometric fixed point map by following the elimination process described in Lemma \ref{lema:generated}, with the order $u^{-1}< u^{-1}_i< u^{-1}_{p-i}< u^{-1}_{p-i+1}< u_{p-1}^{-1}$. 
\end{enumerate}
\end{proof}

\begin{remark}
The right hand side of \eqref{30} is congruent to 
\[(i-1)^{-1}i^{-1}-(i-1)^{-1}+i^{-1}=(i-1)^{-1}(i^{-1}-1)+i^{-1}\]
modulo $p$. Since $i\neq 1$, multiplying the previous expression by $i-1$ gives
\[i^{-1}-1+i^{-1}(i-1)\]
which is congruent to $0$ modulo $p$. Therefore, we see directly that the right hand side of \eqref{30} is indeed divisible by $p$.
\end{remark}

Recall that we have a restriction map $\res: \Omega^{C_p}_\ast\to MU_\ast$, which is given by forgetting the action on a $C_p$-manifold. Using the relations \eqref{eq:24}, and \eqref{25}, as well as the fact that $\res(q_1)=p$, we get 
\begin{align*}
    p\cdot \res(d_{l,j}^{(i)}) & =pt_{l,j}^{(i)} \\
    p\cdot \res(q_k) & =pc_k.
\end{align*}
Since $MU_\ast$ is a domain, we obtain
\begin{align*}
    \res(d_{l,j}^{(i)})& =t_{l,j}^{(i)}\\
    \res(q_k) & = c_k.
\end{align*}
We have the following.
\begin{prop}\label{prop:gamma}
    There is an additive operation $\Gamma: \Omega^{C_p}_\ast \to \Omega_{\ast+2}^{C_p}$ given by
    \begin{align*}
        \Gamma(q_j) & =\res(q_j)(d_{0,0}^{(p-1)}-q_2)+q_{j+1}\\
        \Gamma(d_{l,j}^{(i)})& =\res(d_{l,j}^{(i)})(d_{0,0}^{(p-1)}-q_2)+d_{l,j+1}^{(i)} \\
        \Gamma(d_{I,j}) & = \res(d_{I,j})(d_{0,0}^{(p-1)}-q_2)+d_{l_1,1}^{(i_1)}d_{l_{2},0}^{(i_{2})}\cdots d_{l_{n-1},0}^{(i_{n-1})}d_{l_n,j}^{(i_n)}+\res(d_{l_1,0}^{(i_1)})d_{l_2,1}^{(i_1)}\cdots d_{l_{n-1},0}^{(i_{n-1})}d_{l_n,j}^{(i_n)}+ \\
        & \res(d_{l_1,0}^{(i_1)}d_{l_2,0}^{(i_2)})d_{l_3,1}^{(i_1)}\cdots d_{l_n,j}^{(i_n)}+\cdots + \res(d_{l_1,0}^{(i_1)}d_{l_2,0}^{(i_2)}\cdots d_{l_{n-1},0}^{(i_{n-1})}) d_{l_n,j+1}^{(i_n)}
    \end{align*}
where  $d_{I,j}=d_{l_1,0}^{(i_1)}\cdots d_{l_{n-1},0}^{(i_{n-1})}d_{l_n,j}^{(i_n)}$. For $x\in \Omega^{C_p}_\ast$, this operation satisfies 
\[\kappa(\Gamma(x))=u^{-1}\kappa(x)+\res(x)u^{-1}_{p-1}.\]
\end{prop}

\begin{proof}
By Corollary \ref{corolario}, $\Omega^{C_p}_\ast$ is a free $MU_\ast$-module with basis given by the elements $q_j$ and the monomials $d_{I,j}$. Therefore, defining $\Gamma$ on these elements is enough to obtain an additive operation. It is also enough to check the desired property on $q_j$ and $d_{I,j}$. For $q_j$ we have 
\begin{align*}
 \kappa(\Gamma(q_j)) & = \res(q_j)(\kappa(d_{0,0}^{(p-1)})-\kappa(q_2))+\kappa(q_{j+1}) \\
& = c_j(u^{-1}_{p-1}-(p-1)u^{-1}+pu^{-1})-\sum_{k=1}^{j+1}c_{j+1-k}u^{-k}\\
&= c_ju^{-1}_{p-1}+c_ju^{-1}-c_ju^{-1}-u^{-1}\sum_{k=1}^j c_{j-k}u^{-k}\\
&= u^{-1}\kappa(q_j)+\res(q_j)u^{-1}_{p-1}.
\end{align*}
A similar argument shows the desired property for $d_{l,j}^{(i)}$. The key property is relation \eqref{eq:relations-d}), which gives
\[\kappa(d_{l,j+1}^{(i)})=-t_{l,j}^{(i)}u^{-1}+u^{-1}\kappa(d_{l,j}^{(i)}).\]
For $d_{I,j}$, we have
\begin{align*}
    \kappa(d_{I,j}) &= \res(d_{I,j})(u^{-1}_{p-1}+u^{-1})+\kappa(d_{l_1,1}^{(i_1)})\kappa(d_{l_{2},0}^{(i_{2})})\cdots \kappa(d_{l_{n-1},0}^{(i_{n-1})})\kappa(d_{l_n,j}^{(i_n)})\\
    &+ t_{l_1,0}^{(i_1)}\kappa(d_{l_2,1}^{(i_1)})\cdots \kappa(d_{l_n,j}^{(i_n)}) \\
    & +  t_{l_1,0}^{(i_1)}t_{l_2,0}^{(i_2)}\kappa(d_{l_3,1}^{(i_1)})\cdots \kappa(d_{l_n,j}^{(i_n)})\\
    & \ldots + t_{l_1,0}^{(i_1)}\cdots t_{l_{n-1},0}^{(i_{n-1})} \kappa(d_{l_n,j+1}^{(i_n)}) 
\end{align*}
\begin{align*}
    & = \res(d_{I,j})(u^{-1}_{p-1}+u^{-1})+(-u^{-1}t_{l_1,0}^{(i_1)}+u^{-1}\kappa(d_{l_1,0}^{(i_1)}))\kappa(d_{l_{2},0}^{(i_{2})})\cdots \kappa(d_{l_{n-1},0}^{(i_{n-1})})\kappa(d_{l_n,j}^{(i_n)})\\
    & + t_{l_1,0}^{(i_1)}(-u^{-1}t_{l_2,0}^{(i_2)}+u^{-1}\kappa(d_{l_2,0}^{(i_2)}))\cdots \kappa(d_{l_n,j}^{(i_n)})\\
    & + t_{l_1,0}^{(i_1)}t_{l_2,0}^{(i_2)}(-u^{-1}t_{l_3,0}^{(i_3)}+u^{-1}\kappa(d_{l_3,0}^{(i_3)}))\cdots \kappa(d_{l_n,j}^{(i_n)})\\
    & \ldots + t_{l_1,0}^{(i_1)}\cdots t_{l_{n-1},0}^{(i_{n-1})} (-u^{-1}t_{l_n,j}^{(i_n)}+u^{-1}\kappa(d_{l_n,j}^{(i_n)})) 
\end{align*}
\begin{align*}
    & = \res(d_{I,j})u^{-1}_{p-1}+u^{-1}\kappa(d_{I,j})+ \res(d_{I,j})u^{-1}-u^{-1}t_{l_1,0}^{(i_1)}\kappa(d_{l_{2},0}^{(i_{2})})\cdots \kappa(d_{l_{n-1},0}^{(i_{n-1})})\kappa(d_{l_n,j}^{(i_n)})\\
    & + u^{-1}t_{l_1,0}^{(i_1)}\kappa(d_{l_2,0}^{(i_2)})\cdots \kappa(d_{l_n,j}^{(i_n)})-u^{-1}t_{l_1,0}^{(i_1)}t_{l_2,0}^{(i_2)}\kappa(d_{l_3,0}^{(i_3)})\cdots \kappa(d_{l_n,j}^{(i_n)})\\
    & + u^{-1}t_{l_1,0}^{(i_1)}t_{l_2,0}^{(i_2)}\kappa(d_{l_3,0}^{(i_2)})\cdots \kappa(d_{l_n,j}^{(i_n)})-u^{-1}t_{l_1,0}^{(i_1)}t_{l_2,0}^{(i_2)}t_{l_3,0}^{(i_3)}\kappa(d_{l_4,0}^{(i_4)})\cdots \kappa(d_{l_n,j}^{(i_n)})\\
    & \ldots  + u^{-1} t_{l_1,0}^{(i_1)}\cdots t_{l_{n-1},0}^{(i_{n-1})}\kappa(d_{l_n,j}^{(i_n)})-\res(d_{I,j})u^{-1}
\end{align*}
\begin{align*}
    &= u^{-1}\kappa(d_{I,j})+\res(d_{I,j})u^{-1}_{p-1}.
\end{align*}
\end{proof}

The operation $\Gamma$ in Lemma \ref{lema:kos} is defined for classes $[M]\in \Omega^{C_p}_\ast$ which admit a representative $M$ whose $C_p$-action extends to an $S^1$-action. On the other hand, the operation $\Gamma$ in the previous proposition is defined for all $x\in \Omega^{C_p}_\ast$. Since both operations have the same image under the geometric fixed points map, they coincide module $q_1$ whenever the geometric operation is defined. Together with Lemma \ref{lemma:last}, the previous proposition gives interpretations, in terms of our generators, of the manifolds in Theorem \ref{teo-geo-generators}.
   
\bibliographystyle{amsxport}
\bibliography{references}
\end{document}